\documentclass[12pt]{amsart}
\usepackage{amsmath,amssymb,amsfonts,amsthm,bbm}
\usepackage[latin1]{inputenc}
\usepackage{amscd, graphics,url,xcolor}
\usepackage{color}
\usepackage{hyperref}
\usepackage[margin=2.5cm]{geometry}
 \usepackage{epstopdf}
\numberwithin{equation}{section}

\newtheorem{thm}{Theorem}[section]
\newtheorem{lem}[thm]{Lemma}
\newtheorem{prop}[thm]{Proposition}
\newtheorem{cor}[thm]{Corollary}
\theoremstyle{definition}
\newtheorem{dfn}[thm]{Definition}
\theoremstyle{remark}
\newtheorem{rmk}[thm]{Remark}

\theoremstyle{remark}
\newtheorem{example}[thm]{Example}

\newcommand{\cw}{{\mathcal W}}
\newcommand{\ce}{{\mathcal E}}

\newcommand{\ca}{{\mathcal A}}

\newcommand{\cf}{{\mathcal F}}

\newcommand{\co}{{\mathcal O}}

\newcommand{\cl}{{\mathcal L}}

\newcommand{\C}{\mathbb{C}}

\newcommand{\G}{\mathbb{G}}

\newcommand{\N}{\mathbb{N}}

\newcommand{\R}{{\mathbb{R}}}

\newcommand{\T}{{\mathbb{T}}}

\newcommand{\Z}{\mathbb{Z}}

\newcommand{\beq}{\begin{equation}}
\newcommand{\eeq}{\end{equation}}

 \newcommand{\vw}{{\vec \omega}}

\newcommand{\vh}{{\vec h}}

\newcommand{\vn}{{\vec n}}

\newcommand{\vt}{{\vec t}}

\begin{document}

\title[Directional Ergodicity and Weak Mixing]{Directional Ergodicity, Weak Mixing and Mixing for $\Z^d$- and $\R^d$-Actions}
\author[Robinson]{E. Arthur Robinson Jr}
\address[E. A. Robinson]{Department of Mathematics, George Washington University, Washington, DC 20052}
\author[Rosenblatt]{Joseph Rosenblatt}
\address[J. Rosenblatt]{Department of Mathematics,  University of Illinois at Urbana-Champaign, Urbana, IL 61801}
\author[\c Sah\. in]{Ay\c se A. \c Sah\.in }
\address[A. \c Sahin]{Department of Mathematics and Statistics, Wright State University, Dayton, OH 45435}
\date{November 12, 2022}
\maketitle

\begin{abstract} 

For a measure preserving   
$\Z^d$- or $\R^d$-action $T$, on a Lebesgue probability space $(X,\mu)$, and 
a linear subspace $L\subseteq\R^d$, we define notions of direction $L$ ergodicity, weak mixing, and strong mixing.
For $\R^d$-actions, it is clear that these direction $L$ properties 
should correspond to the same properties for the restriction of $T$ to $L$. 
But since an arbitrary $L\subseteq\R^d$ 
does not necessarily correspond to a nontrivial subgroup of $\Z^d$,  a different appzroach is needed
for $\Z^d$-actions. In this case, we define  direction $L$ ergodicity, weak mixing, and mixing  
in terms of the restriction of the unit suspension 
$\widetilde T$ to $L$, but also restricted  to the 
subspace of $L^2(\widetilde X,\widetilde \mu)$ perpendicular to the 
suspension direction.  
For $\Z^d$-actions, we show (as is more or less clear for $\R^d$) 
that these directional properties are spectral properties.
For weak mixing $\Z^d$- and $\R^d$-actions, we show that 
directional ergodicity is equivalent to directional weak mixing.
For ergodic $\Z^d$-actions $T$, we explore the relationship between direction $L$ properties as defined via unit suspensions and embeddings of $T$ in $\R^d$-actions.  Finally, the structure of 
possible sets of non-ergodic and non-weakly mixing directions is determined, and genericity questions are discussed.
\end{abstract}

\bigskip

\noindent{\it In memorium:}
This article is dedicated to the memory of Uwe Grimm, valued colleague, friend, and coauthor.

\section{Introduction}

\bigskip

In the 1980's Milnor  (\cite{Mil}, \cite{Mil2}) introduced the notion of directional dynamics for $\Z^d$-actions in the context of {\em directional entropy}. 
A direction here means an $e$-dimensional linear subspace $L$ of $\R^d$. 
If an $e$-dimensional  direction 
$L$  is ``completely rational" in the sense that it is generated by a subgroup $\Lambda$ isomorphic to 
$\Z^e$, then the directional theory simply reverts to 
the dynamics of the corresponding subgroup action. Milnor's idea is that 
some properties, like entropy, make sense even in arbitrary (i.e., not completely rational)  directions. 
In the years since Milnor's work, several other theories of directional dynamics have appeared, 
most notably, Boyle and Lind's theory of directional expansiveness \cite{BL}.    
Our goal in this paper is to define and study a theory of 
directional ergodicity, directional weak mixing, and directional strong mixing. Our setting throughout this 
paper is free measure preserving actions $T$ of $\Z^d$ or $\R^d$ (or sometimes  a 
locally compact abelian group $G$)
on a Lebesgue probability space $(X,\mu)$.  

There are two approaches in the literature for defining directional properties for a $\Z^d$-action $T$
in an arbitrary direction $L$.
The first approach, which follows Milnor's original definition 
of directional 
entropy, is to study transformations $\{T^{\vec n_i}\}$ for vectors 
$\vec n_i\in\Z^d$ that approximate $L$ in some way.   The second approach, which 
is equivalent to the first for directional entropy \cite{PIsrael}, as well as several other directional theories
(see \cite{BL} and  \cite{JSRec}), 
is to use the unit 
suspension $\widetilde T$ of $T$. Since the unit suspension 
of a $\Z^d$-action is an $\R^d$-action, 
the directional theory for $T$ simplifies to studying the restriction 
of $\widetilde T$ to subspaces $L$ of $\R^d$. 
It is this latter approach that we use here to define directional ergodicity, weak mixing and strong mixing 
for $\Z^d$-actions. 

This paper is organized as follows.  After providing some basic definitions in Section~\ref{defandex},
we describe the ergodic properties of restrictions of locally compact abelian group $G$-actions $T$ 
to a closed subgroups $H$. 
These results, mostly known,  depend on duality theory for locally compact abelian groups, and the spectral 
theory of the Koopman representations corresponding to the 
actions. They reduce to properties of the spectral measure on the dual group $\widehat G$. 

In Section~\ref{s_rd} we specialize to actions $T$ of the group is $\R^d$, where the subgroup $H$ 
is a linear subspace $L$ of $\R^d$. 
It is clear that the definition of direction $L$ ergodicity or weak mixing  
in this case should be the same property 
for the restriction of $T$ to $L$.
The main idea is that there is an eigenfunction with eigenvalue $\vec \ell\in L=\widehat L$, 
for $T$ restricted to $L$, if and only if the spectral measure $\sigma^T_0$ for $T$ (off the constant functions)  
gives positive measure 
to $L^\perp+\vec\ell\subseteq\R^d$. In this case, we call $L^\perp+\vec\ell$ a ``wall'' for $\sigma^T_0$.
These ideas 
are closely related to results of Pugh and Shub \cite{Pugh} from 1971,
who studied the ergodicity of a single map $R=T^{\vt_0}$ 
in an  
 $\R^d$-action $T$ (which, of course, amounts to restricting $T$ to a  cyclic subgroup).
It turns out to be easy to find weak mixing $\R^d$-actions $T$ that are 
weak mixing in every direction (for example, mixing action $T$, like Bernouli actions, 
and $\R^d$-actions  $T$ with minimal self joinings).
Similarly, ergodic but not weak mixing $\R^d$-actions $T$ always have both ergodic and 
non ergodic directions, but no 
direction can be weak mixing. 
The most interesting case is $\R^d$-actions $T$
that are weak mixing but not mixing.
In  Section~\ref{s:rdstruct} we provide some Gaussian 
examples  with a variety of directional properties, but 
it is remarkable how little is known in general.

Section~\ref{s_zd} begins the main work of this paper. As indicated, 
we define direction $L$ ergodicity and weak mixing for a $\Z^d$-action $T$ in terms 
of the restriction of the unit suspension $\widetilde T$ to an arbitrary subspace (i.e., a direction) $L\subseteq\R^d$. 
However, we immediately encounter a 
new technical issue:
unit suspensions $\widetilde T$ always have 
non-ergodic directions $L$ having
nothing to do with the dynamics of $T$. To deal with this, we 
define direction $L$ ergodicity and weak mixing 
as the absence of eigenfunctions for $\widetilde T$ 
restricted to $L$ 
belonging to the subspace 
of $L^2$ perpendicular to the suspension direction
(Definition~\ref{Zddef1}).  We
show (Corollary~\ref{arespectral})  
that directional ergodicity and weak mixing are spectral properties of $T$ (not just $\widetilde T$), 
which implies  
they are isomorphism invariants for the $\Z^d$-action $T$. 

To justify our definitions of these directional properties, we prove two things:
First, when $L$ is a completely rational direction, 
ergodicity and weak mixing 
for a $\Z^d$-action $T$ in a direction $L$ 
agree with ergodicity and weak mixing for 
$T$ restricted to $\Lambda=L\cap\Z^d$ (Theorem~\ref{p:equivdef}).
Second, in the case that $T$ embeds in an $\R^d$-action 
$\overline T$, ergodicity and weak mixing 
for $T$ in the direction $L$  agree with directional ergodicity and weak mixing for 
$\overline T$ restricted to $\Lambda$ (Theorem~\ref{t:ergemb} and
(Corollary~\ref{t:wmemb}). So our directional definitions are consistent 
with more standard definitions when they make sense.
But our directional theory is more general.

Section~\ref{s:rdstruct} begins by showing that for an ergodic action $T$ of $\R^2$  or $\Z^2$, 
the set of non-ergodic or non-weak mixing directions is always countable. 
For $d>2$ we obtain a similar countability result, but the exact statement
(Theorem~\ref{whatdirections}) is complicated by the fact that if $L$ is
non-ergodic (or non-weak mixing), so are all of its subspaces.   
Conversely, 
we also show that any set of directions that satisfy the conclusion of 
Theorem~\ref{whatdirections} can be realized by a 
weak mixing $\Z^d$-action $T$ 
(Theorem~\ref{t:cecomplete}). In particular for $d=2$, any countable set of directions can be  the
 set of non-ergodic or non-weak mixing directions
This construction uses Gaussian actions.    
These countability results are reminiscent of 
Pugh and Shub \cite{Pugh}, who show that for an
ergodic $\R^d$ action $T$, the set of $\vt_0\in\R^d\backslash\{\vec 0\}$ so that transformation
$T^{\vt_0}$ is not ergodic is a countable union of affine subsets in $\mathbb R^d$.

In Section~\ref{six} we wrap up with some additional results. 
First, we prove (Theorem~\ref{p:ergiswm}) that for a weak mixing 
$\Z^d$- or $\R^d$-action
$T$, directional ergodicity implies directional weak mixing. 
Next, we define directional strong mixing and
show that a strong mixing $\Z^d$-action is strong mixing in every direction.
Finally, we turn to questions of genericity.  
Ryzhikov showed \cite{Ryz} that a generic $\Z^d$-action $T$ is weak mixing and 
embeds in an $\R^d$-action $\overline T$ which is (in our terminology) 
weak mixing in every direction. 
 Since our Corollary~\ref{t:wmemb} says that whenever $T$ embeds in an $\R^d$-action $\overline T$,
 then $T$ and $\overline T$ have the same 
weak mixing directions, it
follows that weak mixing in every direction is a generic property for $\mathbb Z^d$ actions $T$.  
We end the paper with a direct proof of this fact.
  
The authors would like to acknowledge Vitaly Bergelson who initially posed to us the question of how to define directional ergodicity and weak mixing for $\mathbb Z^d$ actions.  We also thank Robert Kaufman for several valuable suggestions regarding Rajchman measures and Kronecker sets in the plane.  

\section{Basic definitions and motivating examples}\label{defandex}

In this section we remind the reader of basic definitions from the spectral theory of ergodic actions, establish notation, and  provide details for some basic examples that motivate the results in the paper.  We refer the reader 
to \cite{Glasner} or \cite{KT} for omitted details.  

\subsection{Group and subgroup actions} \label{s:subgroups}

Let $G$ be a second countable locally compact 
abelian group (from now on we simply call $G$ a {\em group}, and use additive notation). 
We will consider free measurable and measure preserving $G$-actions $T=\{T^g\}_{g\in G}$ on a Lebesgue probability space $(X,\mu)$.  
In particular, $T$ consists of a map $(g,x)\mapsto T^g x:G\times X\to X$ that 
is measurable with respect to the Borel sets in $G$.  Moreover, for each $g\in G$, 
$T^g:X\rightarrow X$ is a measure preserving transformation, the map
$g\mapsto T^g$ is continuous in the weak topology, and $T^{g}\left(T^{h}(x)\right)=T^{h}\left(T^{g}(x)\right)=T^{g+h}(x)$
for all $g,h\in G$.   
 We often just refer to $T=\{T^g\}_{g\in G}$ on  $(X,\mu)$ as {\em the $G$-action $T$}. 

Let $H\subseteq G$ be a closed subgroup.    
The {\em restriction} of $T$ to $H$, denoted by 
$T^{|H}$, is the $H$-action on $(X,\mu)$ defined by $T^{|H}=\{T^h\}_{h\in H}$.  
 
In this paper, the group  $G$ will usually be either $\Z^d$ or $\R^d$ or a 
nontrivial closed subgroup of one of these. 
For $G=\R^d$, the subgroups that will mostly interest us will be the the 
linear subspaces $L\subseteq\R^d$.

\begin{dfn} For $0<e<d$
an $e$-dimensional linear subspace $L\subseteq\R^d$ is called
an {\it $e$-dimensional direction in $\R^d$}.
\end{dfn}

The set of all $e$-dimensional directions in $\R^d$
is an $e(d-e)$-dimensional compact manifold, called the {\em Grassmanian}, and denoted 
by $\G_{e,d}$ (see \cite{Grass}).  We define $\G_{0,d}=\{\vec 0\}$ and $\G_{d,d}=\{\R^d\}$, and  we
write $\G_d:=\cup_{e=0}^{d}\G_{e,d}$.  

The nontrivial subgroups $H\subseteq\Z^d$ are all closed and 
satisfy $H\sim\Z^e$  for some $0<e\le d$. Letting $L=\text{span}(H)\subseteq\R^d$ 
we have that $L\subseteq\R^d$ is an 
$e$-dimensional linear subspace (that is, $L\in\G_{e,d}$).

\begin{dfn}\label{d:rational}
We call  $L\in\mathbb G_{e,d}$ 
a {\it completely rational direction} if there exists $\Lambda\subseteq\Z^d$ with $\Lambda\sim\mathbb Z^e$ so that $\Lambda=L\cap\Z^d$.  A direction
$L$ is called an {\it irrational direction} if $L\cap\Z^d=\emptyset$.
\end{dfn}
Note that in the case $d=2$ if $L$ is not completely rational then it is irrational, but this is not true for $d>2$.  

\subsection{Duals}\label{s:duals}

The {\em dual group}  $\widehat G$ of a group $G$ is the set of {\em characters} of the group, namely the set of all continuous homomorphisms
$\gamma:G\to S^1=\{z\in\C:|z|=1\}$ with pointwise multiplication. With the compact open topology, $\widehat G$ is also a second countable locally compact abelian group.  

For a closed subgroup $H\subseteq G$, 
the {\it annihilator} of $H$ is defined by
$H^\perp:=\{\gamma\in\widehat G:\gamma(h)=1\text{ for all }h\in H\}\subseteq\widehat G$.
One has $\widehat H=\widehat G/H^\perp$, and $H^\perp=\widehat{G/H}$.
We let $\pi:\widehat G\rightarrow\widehat G/H^\perp=\widehat H$ denote the canonical projection.

We identify $\widehat \R^d=\R^d$, and 
for a linear subspace
$L\subseteq\R^d$, we identify the dual $\widehat L$ of  with $L$ itself
by identifying  
$\vec \ell\in  L$ with the character
\begin{equation}\label{char}
{\vec \ell}(\vec t\,)=e^{2\pi  i\, (\vec \ell\cdot\vec t\,)}.
\end{equation} 
In this case, $L^\perp$ is just the perpendicular subspace.

The dual group of $\Z^d$ is $\T^d=\R^d/\Z^d\cong[0,1)^d$, where 
$\vec a\in \T^d$ is identified with the character
\begin{equation}\label{defa}
{\vec a\,}(\vec n\,)= e^{2\pi  i (\vec a\cdot\vec n\,)}.
\end{equation}

If $\Lambda\subseteq\Z^d$ is the maximal subgroup in an 
$e$-dimensional completely rational direction $L\subseteq\R^d$  
then (\ref{defa}) implies 
that $\Lambda^\perp=\pi(L^\perp)\subseteq\T^d$, where
$\pi:\R^d\to\T^d=\R^d/\Z^d$ is the canonical projection. In this case,
$\Lambda^\perp\subseteq\T^d$ is a closed $(d-e)$-dimensional sub-torus, and one has
$\widehat\Lambda=\widehat G/\Lambda^\perp\sim\T^e$.
In the non-maximal case $H\subseteq\Z^d$ one has $H^\perp=\pi(L^\perp)+F$ 
where $L=\text{span}(H)$ and $F\subseteq\pi(L)$ is a finite subgroup.
Geometrically, $H^\perp\subseteq\T^d$ is a finite union of 
closed $(d-e)$-dimensional subtori, and $\widehat H=\widehat\Lambda/\widehat F\sim\T^e$.

\subsection{Eigenvalues and the spectral type}

For a $G$-action $T$, we say a nonzero complex valued
$f\in L^2(X,\mu)$, 
is an {\em eigenfunction} corresponding to {\em eigenvalue}
$\gamma\in\widehat G$ if
$f(T^g x)=\gamma(g)f(x)$
for all $g\in G$ and $\mu$ a.e. $x\in X$.
In particular, an {\em invariant function} $f\in L^2(X,\mu)$, i.e. $f(T^g x)=f(x)$  for all $g\in G$,
is an eigenfunction for eigenvalue $0\in\widehat G$. 
The set $\Sigma_T\subseteq\widehat G$ of eigenvalues of $T$ is called the {\em point spectrum} of $T$.

A $G$-action 
$T$ is {\em ergodic} if the constant functions  
are the only invariant functions. A $G$-action $T$ is {\em weak mixing}
if the constant functions
are the only eigenfunctions. 
The following is standard.

\begin{lem}\label{easyergodic} 
Let $T$ be an ergodic $G$-action. Then
\begin{enumerate}
\item all eigenvalues are simple (each eigenfunction $f$ is unique up to a constant multiple), 
\item all eigenfunctions $f$ have 
constant absolute value $\mu$ a.e. (up to a constant multiple we may assume $|f|=1$), and
\item the set $\Sigma_T\subseteq\widehat G$  
is an (at most) countable subgroup.  
\end{enumerate}
\end{lem}

The {\em Koopman representation} for  a $G$-action $T$
is  the unitary representation of $G$ on the (complex) separable Hilbert space  
$L^2(X,\mu)$, $U_T=\{U_T^{g}\}_{g\in G}$, defined 
 by $(U^{g}_T f)(x):=f(T^{g}x)$. 
 For any $f\in L^2(X,\mu)$ one defines the correlation function 
 \begin{equation}\label{newlabel}
 \varphi_f(g):=(U_T^g f,f)=\int_X f(T^{g}x)\overline{f(x)}\,d\mu.
 \end{equation}
 The function $\varphi_f$ is positive definite on $G$, so by Bochner's Theorem, there exists a unique 
 finite Borel measure $\sigma_f$ on $\widehat G$ with $\varphi_f$ as its Fourier transform 
 (see \cite{KT}): 
 \begin{equation}\label{smgen}
 \varphi_f(g)=
 \int_{\widehat G}\overline{\gamma(g)}d\sigma_f(\gamma).
 \end{equation}
 Assuming $T$ is implicitly known, we call $\sigma_f$ the {\em spectral 
measure} for $f$.   

Many of our arguments depend on consequences of the {\em spectral 
theorem for unitary representations of $G$}. 
Most of what we will need is included in the following proposition.
 
\begin{prop}[see e.g., \cite{KT}]\label{sst}
Suppose $T$ is a $G$-action on $(X,\mu)$ and let $U_T$ be the 
corresponding 
Koopman 
representation of $G$ on $L^2(X,\mu)$. 
Let $U_T|_{\mathcal F}$  be the restriction of $U_T$ to 
the closed $U_T$-invariant subspace
$\mathcal F\subseteq L^2(X,\mu)$. Then there exists
$f_0\in\mathcal F$ {\rm (}called a {\em function of maximal spectral type} for $U_T|_{\mathcal F}${\rm )} so that for any $f\in\mathcal F$, the corresponding 
spectral measures satisfy $\sigma_f\ll\sigma_{f_0}$. 
The measure $\sigma_{f_0}$, which is unique up to 
Radon-Nikodym equivalence, is called 
a {\em measure} of {\em maximal spectral type}
for $U_T|_{\mathcal F}$.
Conversely, for any $\sigma\ll\sigma_{f_0}$ there is an $f\in \mathcal F$ so that 
$\sigma=\sigma_f$.  Moreover, if $\sigma_{f_1}\perp\sigma_{f_2}$
for $f_1,f_2\in\mathcal F$ then $U_T^g f_1\perp f_2$ for all 
$g\in G$.
\end{prop}
    
By Proposition~\ref{sst} there exists $f_*\in L^2(X,\mu)$ that is a 
function of maximal spectral
type for $U_T$ on $L^2(X,\mu)$. 
We call the measure
$\sigma^T:=\sigma_{f^*}$ 
the {\em spectral measure} for $T$.
Let  $L^2_0(X,\mu)=\{1\}^\perp\subseteq L^2(X,\mu)$. 
By Proposition~\ref{sst} there 
is also is a function (also unique up to equivalence) $f_0\in L^2_0(X,\mu)$ of maximal spectral type 
for the restriction $U_T|_{L^2_0(X,\mu)}$. 
We call $\sigma^T_0=\sigma_{f_0}$  the {\em reduced spectral measure} for $T$. 

If two $G$-actions $T$ and $S$ are metrically isomorphic, then their Koopman 
representations are $U_T$ and $U_S$ are unitarily conjugate (also called spectrally isomorphic), 
and thus by the spectral theorem  (see \cite{KT}), their reduced spectral measures $\sigma_0^T$ 
and $\sigma_0^S$ (and their spectral measures $\sigma^T$ 
and $\sigma^S$) are equivalent. Thus the equivalence of spectral measures and reduced spectral measures 
are isomorphism invariants for $G$-actions $T$.
The following results are well known (see e.g., \cite{KT}).

\begin{lem}\label{hasatom} 
 Let $T$ be a $G$-action. Then 
 $\gamma\in \Sigma_T$  
 if and only if $\sigma^T(\{\gamma\})>0$ {\rm (}i.e., $\gamma$ is an atom{\rm )}.
 \end{lem}
 
 \begin{cor}\label{c:noatomatzero}
 Let $T$ be an ergodic $G$-action. Then
\begin{enumerate}
 \item\label{tosszero}  
$\sigma^T_0=\sigma^T-\sigma^T(\{0\})\delta_0$,  where $\delta_0$ 
denotes unit point mass at $0\in\widehat G$, and 
 \item $T$ is  weak mixing if and only if 
 $\sigma_0$ is non-atomic. 
 \end{enumerate}
 \end{cor} 

For notational convenience we will use $\varsigma$ to denote spectral measures of actions of 
 continuous groups and $\sigma$ to denote spectral measures of actions of discrete groups. 

\subsection{Properties of subgroup actions}\label{ss:subaction}
 The following is a straightforward observation.
 
\begin{prop}\label{supergroups}
Let $T$ be a $G$-action, and  
 $H\subseteq G$ a closed subgroup.  If 
 $T^{|H}$ is ergodic (or weak mixing) then 
 $T$ is ergodic (or weak mixing).  
\end{prop}

So ergodicity passes to supergroup actions.
In particular, if $T$ is not ergodic (or weak mixing) then no restriction $T^{|H}$ of $T$ can be ergodic
(or weak mixing).  
The next few examples 
show that the converse is not generally true.   

\begin{example}\label{prodtype}
Let $T_1$ and $T_2$ be measure preserving $\Z$-actions on  
$(X,\mu)$. Define the {\em product type} 
$\Z^2$-action $T_1\otimes T_2$ by 
\[
(T_1\otimes T_2)^{\vec n}(x_1,x_2)=(T_1^{n_1}x_1,T_2^{n_2}x_2),
\]
where $\vec n=(n_1,n_2)$. Let $\Lambda_1=\{(n_1,0):n_1\in\Z\}$ and 
$\Lambda_2=\{(0,n_2):n_2\in\Z\}$ be the vertical and horizontal 
subgroups of $\Z^2$. Clearly the restrictions 
$T_1^{|\Lambda_1}$ and  $T_2^{|\Lambda_2}$
are not ergodic since they correspond to the 
$\Z$-actions $T_1\times Id$ and $Id\times T_2$.
However, it is easy to see that $T_1\otimes T_2$  is ergodic 
(or weak mixing) if and only if  $T_1$ and $T_2$ are ergodic (or weak mixing).
\end{example}

\begin{example}\label{ex:BWtype}
The following example, Bergelson and Ward \cite{BWEx}, is a 
weak mixing $\Z^2$-action $\mathring{T}$ so that $\mathring{T}^{|\Lambda}$ is not ergodic for any subgroup 
$\Lambda=\{k\vec n:k\in\mathbb Z\}$, for $\vn\in\Z^2\backslash\{\vec 0\}$.  

For a 
weak mixing $\Z$-action $T$ on $(X,\mu)$,
let $(\mathring{X}, \mathring{\mu})=\prod_{i=1}^\infty (X,\mu)$, and for 
an enumeration $\{\vec m_i\}_{i\in\Z}$
of $\Z^2\backslash\{\vec 0\}$, define a $\Z^2$-action
$\mathring{T}$ on $\mathring{X}$ by  
\[
\mathring{T}^{\vec n}\left(\prod_{i=1}^\infty x_i\right)=\prod_{i=1}^\infty T^{\vec n\cdot\vec m_i}(x_i).
\]
Note that $\mathring{T}$ is weak mixing as in 
Example~\ref{prodtype}, but
for each subgroup $\Lambda$ the 
transformation $\mathring{T}^{\vn}$ is not ergodic because it acts 
as the identity on factor 
with $\vec m_i\perp\vn$. 
In the terminology of Section~\ref{s_zd},  $T$  every completely rational direction is a is non ergodic direction.
One can show, moreover,  that  
 $T$ is weak in every irrational direction $L$.
 For other similar examples, see Example~\ref{gbw}
\end{example}

Here is a simple way that 
ergodicity or weak mixing can fail for subgroup actions.

\begin{prop}\label{p:easyobs}
Let $T$ be a $G$-action that is not weak mixing, and let  
$\gamma\in\Sigma_T\backslash\{0\}$.  If there exists a subgroup $H\subseteq G$ so that $\gamma\in H^\perp$ then 
$T^{|H}$ is not ergodic.  
\end{prop}

\begin{proof}
Let $f$ be an eigenfunction for $\gamma$.  Then $\gamma\in H^\perp$ means 
$\gamma(h)=1$ for all $h\in H$, which implies  $f(T^h(x))=f(x)$
for all $h\in H$. Thus  $f$ is an invariant function for $T^{|H}$.
The second statement is just the observation at the beginning 
of the section: if  
$T^{|H}$  is weak mixing then $T$ is also weak mixing.
\end{proof} 

\begin{example}
If  $G=\R^d$ than a subgroup $H$ as in Proposition~\ref{p:easyobs} always exists:  if 
$\vec\ell\in\Sigma_T\backslash\{0\}
$ then $H=\langle\vec \ell\rangle^\perp$. 
The situation is different for $G=\Z^d$.  Consider the 
case of $d$ rotations on the circle $\mathbb T$ by rationally independent irrationals, which   
since they commute, define an ergodic 
$\mathbb Z^d$-action $T$ on $\mathbb T$.  
The eigenvalues $\Sigma_T$ satisfy
$\mathbb Z^d\cap\Sigma_T=\{\vec 0\}$ so for any subgroup$H\subseteq\Z^d$, one has 
 $H^\perp\cap\Sigma_T=\{\vec 0\}$.   
Thus all subgroup actions are ergodic.  
Contrast this with the case of 
$\mathbb Z^d$ odometer actions  $T$ (see \cite{Cortez}), all of which satisfy 
$\Sigma_T\subseteq\mathbb Q^d/\mathbb Z^d\subseteq\mathbb T^d$.  
Thus 
there are always subgroups $H\subseteq\mathbb Z^d$ such that $H^\perp\cap\Sigma_T\neq\{\vec 0\}$, so 
there are always non-ergodic subgroup actions.  
\end{example}
 
To go beyond  Proposition~\ref{p:easyobs}
we need to describe the relationship between the spectral measures of $T$ and those of $T^{|H}$.  
We do this here in slightly greater generality that will be useful later.
\begin{prop}\label{p:genrest1}
Let $T$ be a $G$-action on $(X,\mu)$, 
with Koopman representation $U_T$. Let ${U_T|}_{\mathcal F}$ be the restriction of $U_T$ to a closed $U_T$ invariant subspace $\mathcal F\subseteq L^2(X,\mu)$.
For a closed subgroup $H\subseteq G$ let ${U_T|}_{\mathcal F}^H:=(U_{T^{|H}})|_{\mathcal F}$ be 
the unitary representation of $H$ 
on $\mathcal F$ obtained by restricting ${U_T|}_{\mathcal F}$ to $H$. 
If $f_0\in \mathcal F$ 
is a function of maximal spectral type for  ${U_T|}_{\mathcal F}$ 
then $f_0$ is also a function of maximal spectral type for 
${U_T|}_{\mathcal F}^H$. Moreover, 
$\sigma^{T^{|H}}_{f_0}=\sigma^T_{f_0}\circ\pi^{-1}$ 
where $\pi:\widehat G\rightarrow\widehat G/H^\perp=\widehat H$ is canonical projection, and 
$\sigma^T_{f_0}$ on $\widehat G$ and $\sigma^{T^{|H}}_{f_0}$ 
on $\widehat H$ are the corresponding spectral measures.  
\end{prop}

\begin{proof}
Note that for all $h\in H$ and $\gamma\in\widehat G$, $\gamma(h)$ is constant on 
the cosets of $H^\perp$ in $\widehat G$. Let $f
\in \mathcal F$.  By \eqref{smgen} we have
\begin{align*}
\widehat {\sigma_f^T}(h) &=\int_{\widehat G} \overline {\gamma(h)}\, d\sigma_f^T(\gamma) =\int_{\widehat G} \overline {\left(\gamma+ H^\perp\right) (h)}\, d\sigma_f^T(\gamma)\\
&=\int_{\widehat G/H^\perp} \overline {\left(\gamma+ H^\perp\right) (h)}\, d(\sigma_f^T\circ \pi^{-1})(\gamma + H^\perp).
\end{align*}
That is,
\[
\widehat {\sigma_f^T}(h) =
\int_{\widehat G/H^\perp} \overline {\left(\gamma+ H^\perp\right)(h)}\, d(\sigma_f^T\circ \pi^{-1})(\gamma + H^\perp)= 
\int_{\widehat H} \overline {\alpha (h)}\, d(\sigma_f^T\circ \pi^{-1})(\alpha),\]where $\alpha\in\gamma+H^\perp$.
On the other hand
\[\widehat {\sigma_f^T}(h) = \int_X f(T^{h^{-1}}x)\overline {f(x)}\, dm(x) = \widehat {\sigma_f^{T^{|H}}}(h) = \int_{\widehat H} \overline {\alpha(h)}\, d\sigma_f^{T^{|H}}(\alpha).\] 
Hence, for any $f\in \mathcal F$, $\sigma_f^{T^{|H}}$ is the {\em push forward} of $\sigma_f^T$; that is,
$\sigma_f^{T^{|H}} = \sigma_f^T\circ \pi^{-1}.$

Now suppose that $\sigma_{f_0}^T$ is a measure maximal spectral type for 
$T$ on $\mathcal F$. 
Let $f\in \mathcal F$ and suppose 
$\sigma_{f_0}^{T^{|H}}(E)=0$ for some 
$E\subseteq\widehat H$. 
Then by our discussion above 
$\sigma_{f_0}^{T}\left(\pi^{-1}(E)\right)=0$.
Since $f_0$ is a function of maximal spectral type for $T$ on $\mathcal F$, 
it follows that 
$0=\sigma_f^{T}\left(\pi^{-1}(E)\right)=\sigma_f^{T^{|H}}(E)$, so 
$\sigma_{f_0}^{T^{|H}}$ is a measure of  
maximal spectral type of $T^{|H}$
on $\mathcal F$.   
\end{proof}

\begin{cor}\label{p:genrest}
 Let $\sigma^T$ and $\sigma_0^T$ 
denote the spectral and reduced spectral measures for 
 a $G$-action $T$. For a closed subgroup  $H\subseteq G$ let $\sigma^{T^{|H}}$ and 
$\sigma_0^{T^{|H}}$ be the spectral and reduced spectral measures for 
 the restriction $T^{|H}$.
Then $\sigma^{T^{|H}}=\sigma^T\circ\pi^{-1}$ and $\sigma_0^{T^{|H}}=\sigma_0^T\circ\pi^{-1}$.
\end{cor}

Now we can strengthen Proposition~\ref{p:easyobs}.

\begin{prop}\label{p:subgroupwall}
Let $T$ be a $G$-action, $H\subseteq G$ a closed subgroup, and $T^{|H}$ the restriction of $T$ to $H$.
Then $\gamma=\nu+H^\perp\in \widehat H=\widehat G/H^\perp$
is an eigenvalue for $T^{|H}$  for a non-constant eigenfunction if and only if 
$\sigma_0^T(\nu+H^\perp)>0$. 
\end{prop}

\begin{proof}
By Proposition~\ref{hasatom}, for $\gamma\in\widehat H$ to be an non-constant eigenvalue 
for $T^{|H}$ is equivalent to $\sigma_0^T(\{\gamma\})>0$.
On the other hand by Proposition~\ref{p:genrest} we have
$
0<\sigma_0^{T|^H}(\{\gamma\})
=\sigma_0^T(\pi^{-1}\{\gamma\})=\sigma_0^T(\nu+H^\perp).
$
\end{proof}

\begin{cor}\label{c:subgrp}
Let $T$ be a $G$-action and let $H$ be a closed subgroup.  Then 
\begin{enumerate}
\item $T^{|H}$ is not ergodic if and only if 
$\sigma_0^T(H^\perp)>0$, and
\item  $T^{|H}$ is not weak mixing if and only if 
$\sigma_0^T(\nu+ H^\perp)>0$ for some $\nu\in \widehat G$.  
\end{enumerate}
\end{cor}

\begin{rmk}
Pugh and Shub  showed in \cite{Pugh} that an element $T^h$ of a $G$-action 
$T$ is not ergodic if and only if $\sigma_0^T(\{h\}^\perp)>0$.  
For $G=\R^d$ and $\vec h\in\R^d$, define $\Gamma=\{k\vh:n\in\Z\}$.   Then $T^{\vec h}$ is not ergodic 
as a $\mathbb Z$-action if and only if $T^{|\Gamma}$ is not ergodic.  Additionally, since 
$\vec\ell\cdot\vec h\in\mathbb Z$ if and only if $\vec \ell\cdot k\vec h\in\mathbb Z$ 
for all $k\in\mathbb Z$, we have that $\Gamma^\perp=\{h\}^\perp$, so Pugh and Shub's result follows 
from Corollary~\ref{c:subgrp}.    Note, however, 
that  if $L=\{t\vec h:t\in\mathbb R\}$ is the direction generated by $\vec h$, then 
$L^\perp\subseteq \{h\}^\perp$, but  $L^\perp$ does not necessarily lie in the support of $\sigma_0^T$.  Therefore 
$T^{\vec h}$ can fail to be ergodic, even though $T$ is ergodic in the direction $L$.  
\end{rmk}
 
We will use Corollary~\ref{c:subgrp} extensively below, but first we consider a few simple consequences. 
As we noted above, 
when $T^{|H}$ is ergodic (or weak mixing) then $T$ is too. 
Although the converse is generally false, it is true in the following situation, which includes 
the case where $H\subseteq\Z^d$ is a finite index subgroup.  
  
\begin{prop}\label{p:subgroup1}
If $T$ be a weak mixing $G$-action 
and $H\subseteq G$
a closed subgroup, then $T^{|H}$ is weak mixing if and only if $T^{|H}$ is ergodic.
\end{prop}

\begin{proof}
If $T^{|H}$ is weak mixing then it is ergodic. 
Conversely, suppose $T$ is weak mixing and $T^{|H}$ 
is ergodic.  Let $\gamma\in\widehat H$ be an eigenvalue for $T^{|H}$
and let $\mathcal F_\gamma:=
\{f\in L^2(X,\mu):f\left((T^{|H})^{h} x\right)={\gamma}(h)f(x)\}$
be the eigenspace.
Clearly ${\mathcal F}_{\gamma}$ is closed, and by Lemma~\ref{easyergodic} $1$-dimensional. 
We claim that every $f\in\mathcal F_\gamma$ is constant, which will imply that $\gamma\equiv 1$ (i.e.,
$\gamma=0\in\widehat G$).
Indeed, for $f\in {\mathcal F}_{\gamma}$ we have  
$(f\circ T^{g})(T^h x)=f\left(T^{h}\left(T^g\left(x\right)\right)\right)=
{\gamma}(h)(f\circ T^g)(x)$, so $U_T^g f=f\circ T^g\in{\mathcal F}_{\gamma}$, and thus 
${\mathcal F}_{\gamma}$is $U_T$-invariant. 
It follows that for $f\in\mathcal F_\gamma$, and $g\in G$ there is a constant $\alpha(g)\in \mathbb C$,
$\vert\alpha(g)\vert=1$, such that $f\circ T^g=\alpha(g)f$.   
In addition,
$ \alpha(g_1+g_2)f=f\circ T^{g_1+g_2}=f\circ T^{g_1}\circ T^{g_2}=\alpha(g_1)\alpha(g_2)f$.
Since $g\rightarrow U_T^g$ is weakly continuous, 
$\alpha\in\widehat G$.  In particular, $\alpha$ is an eigenvalue for $T$.  
But $T$ is a weak mixing $G$-action, so $\alpha=1$.  Thus $f$ is constant.  
 \end{proof}

\section{Directional properties for $\R^d$-actions}\label{s_rd}

Consider  an $\R^d$-action $T=\{T^{\vec t}\}_{\vec t\in\R^d}$  on $(X,\mu)$.  
Since every direction $L\in\mathbb G_{e,d}$ corresponds to a subgroup of $\R^d$ it is natural to define directional ergodicity for $T$ in a 
direction $L$ to mean the ergodicity of the subgroup action $T^{|L}$ of $T$ (see Section~\ref{defandex}), and similarly for directional weak mixing. 
Note that even though $L\sim\R^e$, one should not think of $T^{|L}$ as an $\R^e$ action. That would require a choice of basis (see \cite{BL}). 
  
\begin{dfn}\label{dirergwm} 
Let $T$ be an  $\R^d$-action and let $L\in{\G}_{e,d}$ be a direction with $1\le e\le d$.
\begin{enumerate}
\item We say $T$ is {\em ergodic} in the direction $L$ if 
the $L$-action $T^{|L}$ is ergodic.  
We denote the set of ergodic directions for $T$
by $\ce_T\subseteq\G_d$.

\item We say $T$ is {\em weak mixing} in the direction $L$ if 
the $L$-action $T^{|L}$ is weak mixing. 
We denote the set of weak mixing directions for $T$
by $\cw_T\subseteq\G_d$.

\end{enumerate}
\end{dfn}

The following basically restates Proposition~\ref{p:subgroup1} in this case.  

\begin{prop}\label{super1}
If $L,L'\in\G_d$ and $L' \supseteq L$ 
then direction $L$ ergodicity (or weak mixing) for $T$ implies direction $L'$ ergodicity (or weak mixing) for $T$.
In particular, direction $L$ ergodicity (or weak mixing) implies $T$ is  ergodic (or weak mixing).
\end{prop}
  
Now we introduce some convenient terminology. 
Note that a nonzero complex valued 
$f\in L^2(X,\mu)$ is an eigenfunction for $T^{|L}$, with eigenvalue $\vec\ell$, if 
$f(T^{\vec t}x)=e^{2\pi i (\vec \ell\cdot \vec t\,)}f(x)$
for all $\vec t\in L$. 
   We call $f$ a {\em direction $L$ eigenfunction for $T$} and 
$\vec \ell\in L$ a {\em directional $L$ eigenvalue} for $T$
(here we implicitly identify $\widehat L$ with $L$
as explained in Section~\ref{s:duals}).
We call $f$ a {\em direction $L$ invariant function} if 
$\vec\ell=\vec 0$. 
Definition~\ref{dirergwm} can now be restated as follows.

\begin{dfn}\label{d:redundant}
Let $T$ be an  $\R^d$-action and let $L\in{\G}_{e,d}$ be a direction with $1\le e\le d$.
\begin{enumerate}
\item We say $T$ is {\em ergodic} in the direction $L$ if 
it has no non-zero direction 
$L$ invariant functions $f$ in $L^2_0(X,\mu)$. 
\item We say $T$ is {\em weak mixing} in the direction $L$ if 
it has no non-zero direction 
$L$ eigenfunctions $f$ in $L^2_0(X,\mu)$ (or equivalently, no direction $L$ eigenvalues,
$\vec \ell\ne\vec 0$\,).
 \end{enumerate}
\end{dfn}

\subsection{Characterizing directional properties spectrally}\label{ss:rdspectrally}
 
\begin{dfn}\label{wall} 
Let $\varsigma$ be a finite Borel measure on $\R^d$. Let  $L\in \G_{e,d}$,
and let $P=L+\vec \ell$ where $\vec\ell\in L^\perp$. 
\begin{enumerate}
\item 
 If  $\varsigma(P)>0$,
 we say that $P$ is an $e$-dimensional {\em wall} for $\varsigma$. 
 
\item A wall $P$ for $\varsigma$ is called {\em proper} if it 
contains no lower dimensional walls $M\subseteq P$ for $\varsigma$. 
 
\item We call a measure $\varsigma$ an $e$-dimensional {\em wall-measure} if there is an 
$e$-dimensional proper wall $P$ for $\varsigma$ that contains its 
 support. We call $P$ the {\em carrier} of $\varsigma$.
  
\end{enumerate}
\end{dfn}

Note that a $0$-dimensional wall measure 
is an atomic measure with a single atom $\vec \ell\in 
\R^d$ as its carrier, and 
a $d$-dimensional wall measure is a (non-zero) measure with carrier
$\R^d$.

\begin{thm}\label{XX}
Let $T$ be an ergodic $\R^d$-action and $L\in \G_d$. 
\begin{enumerate}
\item\label{i:Pugh} 
$L\in \ce_T$ if and only if $L^\perp$ is not a wall for $\varsigma_0^T$.  
\item
$L\in \cw_T$ if and only if no affine subset
$P=L^\perp+\vec \ell$, $\vec \ell\in L$ is a wall for 
$\varsigma_0^T$. In other words, no affine subset perpendicular to $L$
 is a wall for 
$\varsigma_0^T$.
\end{enumerate}
\end{thm} 

This is just Proposition~\ref{p:subgroupwall} translated into the terminology of this section. 
Since spectral type is an isomorphism invariant (see the discussion before Lemma~\ref{hasatom}) we have the following immediate corollary.  
 
\begin{cor}
Directional ergodicity and directional weak mixing for $\R^d$-actions $T$ are spectral 
properties (i.e., they depend only on  $\varsigma_0^T$). 
It follows that 
they are isomorphism invariants.  
\end{cor}

\section{Directional properties for $\Z^d$-actions}\label{s_zd}
For completely rational 
directions  (see 
Definition~\ref{d:rational}) it makes sense to follow the $\R^d$ case and 
define the directional behavior of $\Z^d$-actions 
as the behavior of the associated subgroup action.  This follows our strategy with $\R^d$-actions.  
In this section we show how to use the unit suspension $\widetilde T$ of a $\Z^d$-action $T$ to extend 
the notions 
of directional ergodicity and weak mixing
to arbitrary directions $L\in\G_{e,d}$ in a way that
agrees with the subgroup action definition when $L$ is completely rational.  
Moreover, these directional properties turn out to be spectral properties of $T$ itself.

\subsection{The unit suspension}
Consider  a $\Z^d$-action $T=\{T^{\vec t}\}_{\vec t\in\R^d}$  on $(X,\mu)$. Let $\lambda$ denote Lebesgue probability measure on  the 
$d$-dimensional torus. $\T^d=[0,1)^d$  
The {\it unit suspension}  $\widetilde T$ of $T$ is the $\R^d$-action on 
$(\widetilde X,\tilde\mu):=(X\times [0,1)^d,\mu\times\lambda)$ defined by
$\widetilde{T}^{\vec t}(x,\vec r\,)=(T^{(\lfloor\vec t+\vec r\rfloor)}x,\{\vec t+\vec r\})$,
where $\lfloor \vec w\rfloor\in\Z^d$ denotes
the component-wise floor, and  $\{\vec w\}:=\vec w-\lfloor\vec w\rfloor\in \T^d$. 
It is well known that  $\widetilde{T}$ 
is ergodic if and only if $T$ is ergodic. However, as the next result shows, it is never a weak mixing 
$\R^d$-action.  

\begin{prop}\label{p:suspratl}
Let $T$ be an ergodic $\Z^d$-action and $\widetilde T$ its unit suspension.  
Define the $\R^d$-rotation action $\{\rho^{\vec t}\,\}_{\vec t\in\R^d}$ on 
$[0,1)^d$ by $\rho^{\vec t}(\vec r\,):=
\{\vec r+\vec t\,\}$.   
Then the projection $p_2:X\times[0,1)^d\rightarrow [0,1)^d$ given by $p_2(x,\vec r)=\vec r$ is a factor 
 map from $\widetilde T$ to $\rho$.   In particular, $\widetilde T$ is not weak mixing and its 
 eigenvalues satisfy $\Sigma_{\widetilde T}\supseteq\Z^d$.
\end{prop}

\begin{proof}
Clearly $\rho^{\vec t}\circ p_2=p_2\circ\widetilde T^{\vec t}$. 
Also  $f(x,\vec r)=e^{2\pi i (\vec n\cdot \vec r\,)}$ is an eigenfunction for $\widetilde T$ with eigenvalue 
$\vec n\in\Z^d\subseteq\R^d$.  
\end{proof}

This rotation factor is responsible for the presence of many non-ergodic directions for $\widetilde T$. 

\begin{prop}\label{p:baddir}
If $L\in\G_{d-1,d}$ is such that $L^\perp\cap\Z^d\neq\{\vec 0\}$ then $L\in\ce_{\widetilde T}^c$.
\end{prop}
\begin{proof}
This follows directly from Propositions~\ref{p:suspratl} and~\ref{p:easyobs}.
\end{proof}

\subsection{Definition of directional properties}  
The non-ergodic directions for $\widetilde T$, 
identified in Proposition~\ref{p:baddir}, mean that we cannot define directional properties of $T$ by simply applying 
Definition~\ref{dirergwm} directly to $\widetilde T$. 
However, since this directional non-ergodicity is a result only of the rotation factor $\rho$, 
it has no connection to the dynamical properties of $T$. To overcome this complication,
 we define a closed 
$\widetilde T$ invariant subspace 
\[
L^2_{\mathcal O}(\widetilde X,\tilde\mu):=\mathcal O^\perp\subseteq L^2(\widetilde X,\tilde\mu),
\]
where
\begin{equation}\label{e:defineo1}
\mathcal O:=\left\{F(x,\vec s\,)=f(\vec s\,):f\in L^2\left([0,1)^d,\lambda\right)\right\}\subseteq L^2(\widetilde X,\tilde\mu).
\end{equation}
The next definition defines directional properties for $\Z^d$-actions in analogy with Definition~\ref{d:redundant}.
\begin{dfn}\label{Zddef1}
Let $T$ be a  $\Z^d$-action on $(X,\mu)$, let $\widetilde T$ 
its unit suspension, and let $L\in{\G}_{e,d}$ 
with $1\le e\le d$.
\begin{enumerate}
\item We say $T$ is {\em ergodic} in the direction $L$ 
if $\widetilde T$ has no non-zero direction 
$L$ invariant functions in $L^2_\co(X,\mu)$. 
We denote the set of ergodic directions for $T$
by $\ce_T\subseteq\G_d$.
 \item We say $T$ is {\em weak mixing} in the direction $L$ if 
$\widetilde T$ has no non-zero direction 
$L$ eigenfunctions in $L^2_\co(X,\mu)$. 
We denote the set of weak mixing directions for $T$
by $\cw_T\subseteq\G_d$.
\end{enumerate}
 \end{dfn}

Here is the analogue of Proposition~\ref{super1} for the $\Z^d$ case.
 
\begin{prop}\label{super2}
If $L,L'\in\G_d$ and $L' \supseteq L$ 
then direction $L$ ergodicity (or weak mixing) for $T$ implies direction $L'$ ergodicity (or weak mixing) for $T$.
In particular, direction $L$ ergodicity (or weak mixing) implies $T$ is  ergodic (or weak mixing).
\end{prop}

This follows from the fact that any direction $L'$ eigenfunction must also be a direction $L$ eigenfunction.

 \subsection{Characterizing directional properties spectrally}   
 Let $U_{\widetilde T}$ be the Koopman representation of $\R^d$ corresponding to the unit suspension 
$\widetilde T$ of a
$\Z^d$-action $T$.
The subspace $L^2_{\co}(\widetilde X,\tilde\mu)\subseteq L^2(\widetilde X,\tilde\mu)$ is $U_{\widetilde T}$ is closed and invariant,
and we let $U_{\widetilde T}|_{L^2_{\co}}$ be the restriction of $U_{\widetilde T}$ to $L^2_{\co}(\widetilde X,\tilde\mu)$.
Proposition~\ref{sst} implies that there is a measure  
$\varsigma_\co^{\widetilde T}$ of maximal spectral type for $U_{\widetilde T}|_{L^2_{\co}}$, unique up to equivalence. 
We call $\varsigma_\co^{\widetilde T}$ the {\it aperiodic spectral measure of $\widetilde T$}.
We can now characterize the directional properties if $T$, 
(Definition~\ref{Zddef1}),
in terms of $\varsigma_\co^{\widetilde T}$. 

\begin{prop}\label{likeXX}
Let $T$ be an ergodic $\Z^d$-action and $L\in \G_d$. Then 
$L\in \ce_T$ if and only if $L^\perp$ is not a wall for  $\varsigma_\co^{\widetilde T}$, and 
$L\in \cw_T$ if and only if no affine subset
$P=L^\perp+\vec \ell$, $\vec \ell\in L$ is a wall for 
$\varsigma_\co^{\widetilde T}$. 
\end{prop} 

As with the nearly identical  Theorem~\ref{XX}, this is just Proposition~\ref{p:subgroupwall} 
translated into the terminology of this section.
The next goal is to describe the directional properties of $T$ in terms of its own reduced
spectral measure $\sigma_0^T$
rather than the aperiodic  spectral measure $\varsigma_\co^{\widetilde T}$ for the unit suspension $\widetilde T$.
To do this we 
show how $\sigma_0^T$
determines $\varsigma_\co^{\widetilde T}$. 
The following lemma is a result from \cite{DL}, adapted to our context.  
 Let  $\pi:\R^d\to\T^d=\R^d/\Z^d$ be the canonical projection,
 defined $\pi(\vec t\,)=\vec t\mod \Z^d$, and let $\sim$ denote equivalence of measures..
 \begin{lem}\label{leman}
$\varsigma_{\mathcal O}^{\widetilde T}\circ\pi^{-1}\sim\sigma_0^T.$
\end{lem}

\begin{proof}
Let $\varsigma^{\widetilde T}$ and $\sigma^T$, respectively, be the 
maximal spectral types for $U_{\widetilde T}$ and $U_T$
on $L^2(\widetilde X,\widetilde \mu)$ and $L^2(X,\mu)$. 
Proposition 2.1 in \cite{DL} states that $\varsigma^{\widetilde T}\circ\pi^{-1}\sim\sigma^T$. 
Now,  $\varsigma^{\widetilde T}_\co$ is obtained from  
$\varsigma^{\widetilde T}$ by removing the atoms at $\Z^d\subseteq\R^d$ 
corresponding the eigenfunctions in $\co$, and $\sigma^T_0$  is obtained from $\sigma^T$
by removing the atom at $\vec 0\in\T^d$ corresponding to the invariant functions. 
The result follows since 
$\pi$ maps $\Z^d\subseteq\R^d$ to $\vec 0\in\T^d$. 
 \end{proof}

 By identifying $\T^d=[0,1)^d$ we can think of 
 $\sigma_0^T$ as a measure on $[0,1)^d$ or as
 finite measure on $\R^d$ with support $[0,1)^d$.
 Taking the latter view, we define a finite Borel measure on $\R^d$ by 
 \begin{equation}\label{average1}
\widetilde{\sigma_0^T}=\sum_{\vec n\in \Z^d} 
a_{\vec n} \cdot(\sigma_0^T\circ\tau^{\vn}),
\end{equation}  
where $a_\vn$ is an arbitrary positive sequence with 
$\sum_{\vec n \in \Z^d}a_{\vec n}=1$
and  $\tau^\vn(\vt\,):=\vt-\vn$ is translation on $\R^d$.
 
\begin{cor}\label{c:corDL}
$\varsigma_\co^{\widetilde T}\ll\widetilde{\sigma_0^T}$.
\end{cor}
\begin{proof}
For $\vn\in\Z^d$ let  $Q_{\vec n}=[0,1)^d+\vec n\subseteq\R^d$.   Using 
Lemma~\ref{leman} we have $\varsigma_\co^{\widetilde T}|_{Q_\vn}\circ\tau^{-\vec n}\ll\sigma_0^T$.  
Then $\varsigma_\co^{\widetilde T}=\sum_{\vn\in\Z^d}\varsigma_\co^{\widetilde T}|_{Q_\vn}$.
\end{proof}

If we assume ergodicity, we can strengthen Corollary~\ref{c:corDL} by using the following lemma.

\begin{thm}\label{unit1}
Let $T$ be an ergodic $\Z^d$-action and let $\widetilde T$ be its unit suspension.  
Then $\varsigma_\co^{\widetilde T}\sim\widetilde{\sigma_0^T}$.
\end{thm}

The proof depends on the follwing well-known lemma.
 
\begin{lem}[see Theorem~3.16 in \cite{KT}]\label{symmetry}
Let $T$ be an ergodic $G$-action and let $\gamma\in \Sigma_T$. 
Define
$\tau^\gamma:\widehat G\to\widehat G$ by 
$\tau^\gamma(\omega)=\omega+\gamma$. Then 
$\sigma^T\sim\sigma^T\circ\tau^\gamma$, where $\sim$ denotes the 
equivalence of measures.
\end{lem}

 \begin{proof}[Proof of Theorem~\ref{unit1}]
 Since $T$ is ergodic, so is $\widetilde T$.  
 By Proposition~\ref{p:suspratl} we have $\mathbb Z^d\subseteq\Sigma_{\widetilde T}$.  
By Lemma~\ref{symmetry}, for all $\vec n\in\mathbb Z^d$
$\varsigma_\co^{\widetilde T}|_{Q_\vn}\sim \varsigma_\co^{\widetilde T}|_{Q_{\vec 0}}$.
 By the definition of $\pi$ and Lemma~\ref{leman} this means for all $\vec n\in\mathbb Z^d$
$\varsigma_\co^{\widetilde T}|_{Q_\vn}\sim \sigma_0^T\circ\tau^{-\vec n}$
 therefore $\varsigma_\co^{\widetilde T}=\sum_{\vn\in\Z^d}\varsigma_\co^{\widetilde T}|_{Q_\vn}\sim  
 \widetilde{\sigma_0^T}$.
 \end{proof}

We are now ready to define directional behavior of a $\Z^d$-action $T$ in terms of its own spectral measure
$\sigma_0^T$.
We begin by 
adapting the definition of a wall measure, Definition~\ref{wall}, 
to the case of measures on $\T^d$. Let 
$\pi:\R^d\to\T^d=[0,1)^d$ defined $\pi(\vt\,)=\vt \mod \Z^d$.   
For $L\in \G_{e,d}$ the projection $\pi(L^\perp)$ 
is a $(d-e)$-dimensional (not necessarily closed) subgroup of $\T^d$.  
For  $\vec \ell\in L$, we refer to $P=\pi(L^\perp+\vec\ell)=\pi(L^\perp)+\pi(\vec \ell)$ as an affine subset in $\T^d$ 
perpendicular to $L$.  

\begin{dfn}\label{wallZ} 
 Let $\sigma$ be a finite Borel measure on  $\T^d$. Let $L\in \G_{e,d}$
so that  $L^\perp\in \G_{d-e,d}$. 
\begin{enumerate}
\item Let $P=\pi(L^\perp)+\vec \ell$ be an affine subset $\T^d$.
If  $\sigma(P)>0$,
 we say that $P$ is a $(d-e)$-dimensional {\em wall} for $\sigma$. The wall 
 $P$ is said to be in the direction $L^\perp$, or 
 perpendicular to $L$.
 
\item A wall $P$ for $\sigma$ is called {\em proper} if $P$ 
contains no lower dimensional walls $Q\subseteq P$ for $\sigma$. 

\item We call a measure $\sigma$ an $e$-dimensional {\em wall-measure} if there is an 
$e$-dimensional proper wall $P$ for $\sigma$ that 
contains its support. 
We call $P$ the {\em carrier} of $\sigma$.
\end{enumerate}
\end{dfn}
 
A $0$-dimensional wall measure on $\mathbb T^d$ is an atomic measure with a single atom 
$\vec \ell\in\T^d$ as its carrier, and 
a $d$-dimensional wall measure is a (non-zero) measure with 
carrier $\T^d$. 
The following theorem is the $\Z^d$ analogue of Theorem~\ref{XX}. 

\begin{thm}\label{XXZ2}
Let $T$ be an ergodic $\Z^d$-action and $L\in \G_{e,d}$.  
Then 
\begin{enumerate}
\item  
$L\in \ce_T$ if and only if $\pi(L^\perp)$ is not a wall for $\sigma_0^T$, and 
\item $L\in \cw_T$ if and only if no affine subset
$P=\pi(L^\perp)+\pi(\vec \ell\,)\subseteq\T^d$, $\vec \ell\in L$,
 is a wall for 
$\sigma_0^T$.
\end{enumerate}
\end{thm}

\begin{cor}\label{arespectral}
Directional ergodicity and directional weak mixing for $\Z^d$-actions $T$ are spectral properties (i.e., they depend only on $\sigma_0^T$). It follows that 
they are isomorphism invariants.  
\end{cor}

\begin{proof}[Proof of Theorem~{\rm \ref{XXZ2}}]
Definition~\ref{Zddef1} and Theorem~\ref{XX} together imply that $L\notin\cw_T$ if and only if there is an affine subset $P=L+\vec\ell$ that is a wall for $\varsigma_{\mathcal O}^{\widetilde T}$.   The result then follows immediately from Theorem~\ref{unit1}.  A similar argument holds for $L\notin\ce_T$.  
\end{proof}
 
 \subsection{Completely rational directions}

In this section we show that, in completely rational directions, 
Definition~\ref{Zddef1} agrees with 
ergodicity and weak mixing of actions restricted to 
the corresponding subgroups.
 
 \begin{thm}\label{p:equivdef}
Let $T$ be an ergodic $\Z^d$-action and let $L\in\mathbb G_{e,d}$ 
be a completely rational direction, with $0<e< d$.  Let $\Lambda=\Z^d\cap L$. Then $T^{|\Lambda}$ is an ergodic (weak 
mixing) subgroup action if and only if $T$ is ergodic (weak mixing) in the direction $L$.
\end{thm}

\begin{proof}
Suppose that  $T$ is not weak mixing in the direction $L$.  Then by Definition~\ref{Zddef1}, there is $F\in L^2_\co(\widetilde X,\widetilde \mu)$, $F\ne 0$, 
and $\vec\ell\in L$.
Then $\widetilde T^{\vec v}(x,\vec r)=e^{2\pi i(\vec\ell\cdot\vec v\,)}F(x,\vec r\,)$ for all 
$\vec v\in L$.  Since $F\ne 0$, there exists $\vec r\in\mathbb T^d$ such that $F|_{X\times\{\vec r\}}\ne 0$.   
So $f(x):=F(x,\vec r\,)$ is a non-zero 
eigenfunction for $T^{|\Lambda}$ with eigenvalue $\pi(\vec\ell\,)\in\pi(L)\subseteq\mathbb T^d$,
(including the case 
$\vec\ell=\pi(\vec\ell\,)=\vec 0$).   Thus if 
$T^{|\Lambda}$ is ergodic or weak mixing, then $T$ is ergodic or weak mixing in the direction $L$. 

While it is possible to prove the converse  by extending the 
eigenfunctions for $T|^\Lambda$  to eigenfunctions for $\widetilde T^{|L}$, we prove the converse a different way: by 
showing it is immediate corollary of our previous work.  Suppose that $T^{|\Lambda}$ is not 
weak mixing.  Then by Theorem~\ref{XX} there is $\vec\ell\in L$ such that $\sigma_0^T(\vec\ell+L^\perp)>0$, 
and by 
Theorem~\ref{XXZ2}, 
$L\notin\mathcal W_T$. 
Setting $\vec\ell=\vec 0$ we have the analogous statement for ergodicity.  
\end{proof}
                
\subsection{Restrictions and embeddings}\label{s:restem}

Let $T=\{T^\vn\}_{\vn\in\Z^d}$ be a $\Z^d$-action on $(X,\mu)$. If 
$\overline T=\{\overline T^{\vec v}\}$ is an $\R^d$-action on $(X,\mu)$,
 so that $T=\overline T|^{\Z^d}$, then we
 call $\overline T$ an {\em embedding} of $T$.  As noted in the introduction, while not every 
 $\Z^d$-action has an embedding
 (and it may not be unique), but the property is generic (see \cite{Ryz}, see also Section~\ref{s:rigidity}).
 In this section we restrict our attention to  
 $\Z^d$-actions $T$ that do have an embedding, 
 and study the relationship between direction $L$ properties of $T$ 
 and those of the restriction $\overline T^{|L}$. 
We begin with the following observation relating embeddings 
to the unit suspension. 
\begin{lem}\label{e:embunit}
Let $T$ be an 
ergodic $\Z^d$-action, 
let $\overline T$ be an 
embedding,  and let 
$\widetilde T$ be the unit suspension.
Then 
$\varsigma_0^{\overline T}\ll\varsigma_0^{\widetilde T}$. 
\end{lem}

\begin{proof}
We identify $\vec s\in\mathbb T^d=[0,1)^d$ with $\vec s\R^d$. 
Then the map 
$\Phi:\widetilde X\rightarrow X$ defined by 
 $\Phi(x,\vec s\,)=\overline T^{\vec s}(x)$
is a factor map between $\widetilde T$ and $\overline T$.  It then follows that the Koopman 
representation $U_{\overline T}$ is unitarily equivalent to a restriction of $U_{\widetilde T}$ to an 
invariant subspace, and by Proposition~\ref{sst}, 
 $\varsigma^{\overline T}\ll\varsigma^{\widetilde T}$.
Since $T$ is ergodic, so are $\widetilde T$ and $\overline T$, so 
$\varsigma_0^{\overline T}\ll\varsigma_0^{\widetilde T}$
(by removing an atom at $\vec 0$ from each).
\end{proof}

 \begin{thm}\label{t:ergemb}
Let $T$ be an ergodic $\Z^d$-action and let 
$\overline T$ be an embedding  
of $T$ in an $\R^d$-action. For any 
$L\in\mathbb G_{e,d}$,  the restriction $\overline T^{|L}$ is ergodic if and only if $T$ is 
ergodic in the direction $L$.
\end{thm}
\begin{proof}
By Theorem~\ref{XX}, $L\notin\ce_{\overline T}$ implies that $L^\perp$ is a 
wall for $\sigma^{\overline T}_0$, namely $\sigma^{\overline T}_0(L^\perp)>0$.   
By Lemma~\ref{e:embunit},  it is also wall for $\varsigma_0^{\widetilde T}$, and 
by Theorem~\ref{XX},  $L\notin\ce_{T}$. 

For the converse, suppose $L\notin\ce_{T}$.    
Then $\sigma^{T}_{0}\left(\pi\left(L^\perp\right)\right)>0$, and since 
$T=\overline T^{|\Z^d}$, 
we have 
$\sigma_0^{T}=\varsigma_0^{\overline T}\circ\pi^{-1}$ by Proposition~\ref{p:genrest} .
Therefore $\sigma_0^{\overline T}\left(L^\perp\right)>0$, and by Theorem~\ref{XX}, $L\notin\ce_{\overline T}$.  
 \end{proof}

If $T$ is a weak mixing $\Z^d$-action, any embedding of $T$ must be a weak mixing as $\R^d$-action.   In this case, 
Theorem~\ref{p:ergiswm} (below) shows ergodicity in a direction is equivalent to weak mixing, and we have the following immediate 
corollary of Theorem~\ref{t:ergemb}. 

\begin{cor}\label{t:wmemb}
Let $T$ be a weak mixing $\Z^d$-action and let
 $\overline T$ denote an embedding of $T$. Then
 $T$ is weak mixing in a direction $L\in{\G}_{e,d}$  if and only if $\overline T^{|L}$ is weak mixing. 
\end{cor}
 
\begin{example}
Here we show
the consequence of removing the ergodicity hypothesis from Theorem~\ref{t:ergemb}.
Let $T$ denote the identity $\Z^d$-action on $\mathbb T^d$. We can embed $T$
both into the identity $\R^d$-action $\overline T_1$ on $\mathbb T^d$, as well as the rotation action of $\R^d$ on 
$\mathbb T^d$ given by  $\overline T_2^{\vec v}(\vec x)=\{\vec x+\vec v\}$.   
Notice that  $\widetilde T$ is an uncountable union of $\R^d$ rotations on $\mathbb T^d$.  Further, for any $A\subseteq\mathbb T^d$ with 
$\mu(A)>0$, the characteristic 
function $\chi_{A\times\mathbb T^d}\in L^2_{\co}(\widetilde X,\widetilde\mu)$ is an $L$ invariant function 
for $\widetilde T$ for any direction $L$. Therefore according to Definition~\ref{Zddef1}, 
$T$ is not ergodic in any direction $L$. 
Turning to the embeddings we see that   $\overline T_1$ is not ergodic in any direction $L$,  while $\overline T_2$ is 
ergodic in any irrational direction.   
\end{example}

\section{The structure and realizations of $\ce_T$ and $\cw_T$} \label{s:rdstruct} 
In this section we turn to the question of what sets can arise as $\ce_T$ and $\cw_T$.  
Since every direction $L$ corresponds to a subgroup of $\R^d$, 
it follows 
$\ce_T=\cw_T=\emptyset$ for $\R^d$-actions $T$ that are not ergodic 
(see Section~\ref{ss:subaction}).  
Analogously, since completely rational directions $L$ correspond to subgroups of $\Z^d$,
 we have that for a non-ergodic $\Z^d$-action $T$, $\ce_T$, and therefore $\cw_T$, 
can contain not contain any  completely rational directions.   
More generally, if $T$ is a non ergodic $\Z^d$-action, then 
\eqref{tosszero} of Corollary~\ref{c:noatomatzero} does not hold, so $\sigma_0^T$ has an atom at $\vec 0$.
Thus,  $L^\perp$ is trivially a wall for every direction $L\in\mathbb G_d$,  
and $\ce_T=\cw_T=\emptyset$.
Therefore for the remainder of the section we focus our attention on ergodic actions.  We use the spectral characterizations of directional properties that we provided in previous sections to prove structure theorems analogous to the results in \cite{Pugh}.  

\subsection{The case $d=2$}
\begin{thm}\label{t:2dimstructure}
If $T$ is an ergodic  $\R^2$- or $\Z^2$-action, then the set 
 $\ce^c_T$ of non-ergodic {\rm  (}$1$-dimensional{\rm )} directions 
  is at most countable. 
  Similarly, if $T$ is weak mixing, then the set  $\cw^c_T$ of non-weak mixing 
{\rm  (}$1$-dimensional{\rm )} directions is at most countable.
\end{thm}

\begin{proof}
By Theorems~\ref{XX} (for $\R^2$) and \ref{XXZ2} (for $\Z^2$), 
for any two non-parallel $L_1,L_2\in\mathcal E_T^c$, $L_1^\perp$ and $L_2^\perp$ are walls for 
$\sigma_0^T$. But $\sigma_0^T(L_1^\perp\cap L_2^\perp)=\sigma_0^T(\{\vec 0\})=0$, 
since $T$ is ergodic,
and therefore the respective wall measures are mutually singular.
By Proposition~\ref{sst},
 the corresponding invariant functions $f_1$ and $f_2$ (modulo constant multiples)
 are orthogonal in $L^2(X,\mu)$.  Since $L^2(X,\mu)$ is separable, there can only be countably many such functions.  
A similar argument holds if
$T$ is weak mixing.
It also follows from Theorem~\ref{p:ergiswm} (below).  
\end{proof}  

\begin{example}\label{pintil}
Theorem~\ref{t:2dimstructure} has the following easy application to the Conway-Radin pinwheel tiling
dynamical system \cite{CR}. This is a uniquely ergodic, weak mixing  $\R^2$-action 
$T$ on the set $X$ of all pinwheel tilings.
The action $T=\{T^\vt\}_{\vt\in\R^2}$ has the property that it is isomorphic to all of its rotations. In particular, 
if $M_\theta$ denotes the $2\times 2$ rotation matrix by $\theta$, then
there is a measure preserving transformation $R_\theta$ on $X$ (it rotates tilings) 
so that $R_\theta T^{M_\theta\vt}=T^\vt R_\theta$.
It follows that if $L$ is a non weak mixing 1-dimensional direction (i.e., there is a wall in direction
$L^\perp$) then so is $M_\theta L$ for any $\theta$. 
But since the number of non weak mixing directions can be at most countable, 
$T$ must weak mixing (and thus ergodic) in every direction.  
\end{example}

\subsection{The case $d>2$}

Propositions~\ref{super1} and \ref{super2} 
show that any result analogous to Theorem~\ref{t:2dimstructure} in higher dimensions will require more care.
We start by introducing some terminology that will 
allow us to describe non-ergodic and non-weak mixing directions efficiently.  
\begin{dfn}
Call a set of directions $\mathcal L\subseteq\G_d$ 
{\em concise} if for all $L,L'\in\mathcal L$ with $L\subseteq L'$ implies $L=L'$.
\end{dfn}

\begin{dfn}
Let $\cl\subseteq\G_d$. A direction $L\in\G_d$ is 
{\em subordinate} to $\mathcal L$ if $L\subseteq L'$ for some $L'\in\mathcal L$.
\end{dfn}

Our main result here is the following.

\begin{thm}\label{whatdirections}
Let $T$ be an ergodic $\R^d$ or $\Z^d$-action.  Then there exist unique, and 
at most countable, concise sets 
$\mathcal N\ce_T\subseteq\G_d$  and $\mathcal N\cw_T\subseteq\G_d$
so that $L\in\ce^c_T$ if and only if $L$ is subordinate to $\mathcal N\ce_T$ and $L\in\cw_T^c$ if and only if $L$ is 
subordinate to  $\mathcal N\cw_T$.
\end{thm}

The proof of Theorem~\ref{whatdirections} requires an extension of the Lebesgue decomposition 
theorem to wall measures.
This idea is implicit in \cite{Pugh}.

\begin{prop}\label{p:ctbl}
Any finite Borel measure $\sigma$ on $\R^d$ or $\T^d$ can be written as a sum
$
\sigma=\sigma_0+\sigma_1+\dots+\dots+\sigma_d
$
where for 
each $e=0,1,\dots,d$, the measure 
$\sigma_e$ is either zero, or 
a sum of at most countably many  
$e$-dimensional wall-measures with distinct (but not necessarily disjoint)
carriers. This decomposition is unique.
\end{prop}

\begin{proof} 
Using Lebesgue decomposition we write  $\sigma=\sigma_0+\sigma'_1$ where $\sigma_0$ is the discrete part of 
$\sigma$.  A standard separability argument yields that $\sigma_0$ is a (possibly trivial) countable sum of atoms. In other words, $\sigma_0$ is 
a countable sum of $0$-dimensional wall measures. 

We proceed in a similar manner by induction. Let
$
\sigma=\sigma_0+\sigma_1+\dots+\sigma_{e-1}+\sigma'_e,
$
where $e<d$, and such that 
for each $0\le k\le e-1$, $\sigma_k$ is either the zero measure or a countable sum of 
 $k$-dimensional wall measures with distinct carriers, and such that $\sigma'_e$ has no walls 
of dimension $0\le k< e$. 

Suppose $\ca_e$ is the set of $e$-dimensional walls for $\sigma'_e$. If $\ca_e$ is empty then $\sigma_e$ is the zero measure.  
If not, we show $\ca_e$ is (at most) countable.  For if we  suppose not, then 
by the pigeonhole principle, there exists $\epsilon>0$ and an 
uncountable subset $\ca'_e\subseteq \ca_e$ with the property that 
$\sigma'_e(P)\ge\epsilon$ for all $P\in \ca_e'$. Then for any  countable subset 
$\ca_e''\subseteq \ca_e'$,
$\sigma'_e \left(\bigcup_{P\in \ca_e''} P \right) = \sum_{P\in \ca_e''}\sigma'_e(P)=\infty$,
where equality follows from  $\sigma'_e(P_1\cap P_2)=0$ for $P_1,P_2\in \ca_e''$
distinct. This is because 
$\sigma'_e$ has no $k$-dimension walls
 for any $k<e$. But this contradicts the fact that $\sigma'_e$
 is a finite measure.

Now we let $\sigma_e$ be the sum of the wall-measures whose distinct carriers are the walls $P\in \ca_e$. By the above, it is a countable sum.

\end{proof}

\begin{proof}[Proof of Theorem~{\rm \ref{whatdirections}}]
Use Proposition~\ref{p:ctbl} to write 
$\sigma_0^T=\sigma_0+\sigma_1+\dots+\dots+\sigma_d$,
and for $0\le i\le d$, let $\ca_i$ denote the (possibly empty) set of distinct carriers $P=K+\vec \ell$, $K\in \G_{i,d}$ and 
$\vec \ell\in K^\perp$, 
for $\sigma_i$.
Define the subset of $\mathbb G_{d-i,d}$ given by the orthocomplements of the subspaces in $\ca_i$, namely
$
\ca_i^\perp=\{L=P^\perp: P\in \ca_{i}\}\subseteq \mathbb G_{d-i,d}.
$
This is the set of non-ergodic $(d-i)$-dimensional directions for $T$ and 
using 
the same argument as Theorem~\ref{t:2dimstructure}
we know that this is a countable set. 
We now define
$
\mathcal N\ce_T=\bigcup_{i=1}^d\left(\ca_i^\perp\,\,
\backslash\bigcup_{0\le k<i}\ca_k^\perp\right).
$
This set is the countable union of all the countable non-ergodic directions with redundancies, which are a 
consequence of Propositons~\ref{super1} and \ref{super2}, having been removed.

We construct $\mathcal N\cw_T$ in a similar way.  Uniqueness of both sets follows from the uniqueness of the orthocomplements restricted to $\mathbb G_d$.

\end{proof}

\begin{cor}\label{whatdirections1}  
If $T$ is ergodic then $\ce_T\neq\emptyset$.  
\end{cor}

The natural question to ask now is which countable, concise subsets of $\G_d$ can be realized as $\mathcal N\ce_T$ for some ergodic $T$?  We begin with the following result which generalizes our discussion from Section~\ref{ss:subaction}.    
\begin{prop}\label{p:nwm}
If $T$ is an ergodic $\R^d$ or $\Z^d$-action that is not weak mixing, then  
$\mathcal N\ce_T\neq\emptyset$ and $\mathcal N\cw_T=\{\R^d\}$, (i.e., $\mathcal W_T^c=\mathbb G_d$).  
 \end{prop}

\begin{proof}
For $\R^d$-actions this is a case of Proposition~\ref{p:easyobs},
since every direction $L$ corresponds to a subgroup action of $\R^d$.   

If $T$ is a $\Z^d$-action that is 
not weak mixing, then $\sigma_0^T$ has atoms 
(corresponding to the nonzero 
eigenvalues). 
Since each atom must lie in 
$\pi\left(K\right)\subseteq\mathbb T^d$ for some $K\in\G_{d-1,d}$, we have $\pi\left(K\right)$ is a wall for $\sigma_0^T$. 
By Theorem~\ref{XXZ2}, this implies $L=K^\perp\in \G_{1,d}$ is a non-ergodic direction for $T$. 
Moreover, for any atom $\vec b$ for  
$\sigma_0^T$, 
and for any $L\in \G_d$, there is an $\vec \ell\in L$ so that 
$\vec b\in \pi(L^\perp+\vec \ell\,)$. Then $\pi(\vec \ell\,)\in\mathbb T^d$ is a  direction $L$ eigenvalue for $T$. 
\end{proof}

\begin{example}\label{pentil}
The Penrose tiling dynamical system $T$ consists of $\R^2$ acting on the set of $X$ all Penrose tilings. 
Penrose tilings can be defined in (at least) three distinct ways: 
by a local rule (a matching rule), by a tiling substitution (an inflation), and in terms of model sets
(see \cite{BG}). The Penrose tiling dynamical system $T$ 
is minimal, uniquely ergodic, and has pure discrete spectrum 
($\sigma_0^T$ is atomic) with 
$\Sigma_T=\Z[e^{2\pi i/5}]$, viewed as a subset of $\R^2$ (see \cite{RRP}).
In particular $T$ is not weak mixing. 
Thus a direction $L$ is ergodic if and only if $L^\perp\cap
\Z[e^{2\pi i/5}]=\{0\}$, and there are no weak mixing directions.  
Almost identical considerations apply to all model set dynamical systems,
all of which have pure discrete spectrum (see \cite{BG}). These considerations also apply to the chair 
tilings \cite{RRTC}, which have pure discrete spectrum 
$\Sigma_T=\Z[1/2]^2$, so in this case, $L$ is an ergodic direction if and only if 
$L^\perp\cap\Z[1/2]^2=\{0\}$, and again there are no weakly mixing directions. 

Like the chair tilings, the table tilings, also discussed in  \cite{RRTC}, have 
$\Sigma_T=\Z[1/2]^2$, 
so again there are no weak mixing directions. But the table tilings  have ``mixed spectrum'' in the sense that  
the spectral measure $\sigma_0^T$ has a nontrivial singular continuous component:
$\sigma'\ll\sigma_0^T$. From this it follows that 
$\ce_T^c\subseteq\{L:L^\perp\cap(\Z[1/2]^2)\backslash\{0\}\ne\emptyset\}$. We 
conjecture that equality holds, which would follow from $\sigma'$ having no walls.
\end{example}

\subsection{Realization}\label{real}  
In this section we  show that any countable concise set
$\cl\subseteq\G_d$ is the set of non-weak mixing (and therefore non-ergodic) directions of a 
weak mixing action of $\Z^d$ and $\R^d$.  

\begin{thm}\label{t:cecomplete}
For any finite or countably infinite concise set $\cl\subseteq\G_d$, 
there is a weak mixing $\R^d$-action $T$ and a weak mixing $\Z^d$-action $S$ with the property that $\cl=\mathcal N\ce_T=\mathcal N\cw_T=\mathcal N\ce_S=\mathcal N\cw_S$.  
\end{thm}

\begin{proof}
We first construct the $\R^d$ example. 
 Let $\mathcal L=\{L_1,L_2,\cdots\}$ be a countable concise set of directions in $\R^d$.   Define the set
$
 \mathcal K=\{K_1, K_2,\ldots\}
$
where $K_i=L_i^\perp$ for each $i$.   Note that $\mathcal K$ is also a countable concise set of subspaces of $\R^d$.  
We will construct an $\R^d$-action $T$ with the property that for any wall $P$ for $\sigma_0^T$ there exists $K\in\mathcal K$ so that $K\subseteq P$.  Since $K^\perp=L_j$ for some $j$, this will imply that $L_j\supseteq P^\perp$, and therefore $P^\perp$ is subordinate to $\mathcal L$.  Thus we will have that $\mathcal L=\mathcal N\ce_T$.

We will use the {\em Gaussian Measure Space Construction} (GMC),  starting with 
a finite Borel measure $\sigma$ on $\R^d$ to obtain a {\em Gaussian $\R^d$-action} 
$T$  on a Lebesgue probability space $(X,\mu)$ with 
$\sigma_T=\exp(\sigma)=\sum_{i=0}^\infty \frac{\sigma^{(n)}}{n!}$ (see \cite{KFS}). Here $\sigma^{(n)}=\sigma * \sigma * \ldots * \sigma$ denotes the $n$-fold convolution power of $\sigma$, $n\ge 1$ and $\sigma^0=\delta_{\vec 0}$. 

If  $\dim(K_i)=r_i$ let $\sigma_{K_i}$ denote $r_i$-dimensional Lebesgue measure on $K_i$ and let $\sigma$ be any finite measure equivalent to $\sum_{i} \sigma_{K_i}$.   The action $T$ will be the Gaussian action with $\sigma_0^T=\exp(\sigma)-\delta_{\vec 0}$. Note that since $\sigma$ is non-atomic, $T$ is weak mixing.

We now identify all proper walls for $\sigma_0^T$.  Each one has to be a proper wall for some 
$
\sigma^{(n)}\sim\sum_{i_1,i_2,\ldots,i_n} \sigma_{K_{i_1}}*\sigma_{K_{i_2}}*\ldots*\sigma_{K_{i_n}}
$
thus proper walls for $\sigma^{(n)}$ must be proper walls for some $\sigma_{K_{i_1}}*\sigma_{K_{i_2}}*\ldots*\sigma_{K_{i_n}}$.  It is easy to see that the
subspace $K=K_{i_1}+K_{i_2}+\ldots+K_{i_n}$ is the only proper wall for $\sigma_{K_{i_1}}*\sigma_{K_{i_2}}*\ldots*\sigma_{K_{i_n}}$ because $K_j$ is the only proper wall for  $\sigma_{K_j}$.  Thus $P$ must be of the form $K_{i_1}+K_{i_2}+\ldots+K_{i_n}$.  Therefore $P$ contains every subspace in $\{K_{i_1},{K_{i_2}},\ldots,{K_{i_n}}\}$ and $P^\perp=L_{i_1}^\perp\cap L_{i_2}^\perp\ldots\cap L_{i_n}^\perp\subseteq L_{i_j}$ for $j=1,\ldots n$.

We now set $S=T|^{\Z^d}$, the restriction of $T$ to $\Z^d$.   By Proposition~\ref{p:genrest} we have 
that $\sigma_0^{S}=\sigma_0^T\circ\pi^{-1}$.  Note that the proper walls for $\sigma_0^S$ are projections of 
the proper walls for $\sigma_0^T$ and the result follows for $\Z^d$.  Alternatively the same construction 
as above can be repeated using measures on $\T^d$ instead of $\R^d$ 
to obtain a $\Z^d$-action with the same properties.  
\end{proof}

\begin{example}\label{gbw}
The idea in Example~\ref{ex:BWtype} provides a recipe for making many 
other examples of $\Z^d$ and $\R^d$-actions
with interesting directional properties. The examples that can be made this way 
have more or less the same directional properties 
as the Gaussian actions constructed in  Theorem~\ref{t:cecomplete}. 
Here is one such example. Choose a $1$-dimensional direction $L_1\subseteq\R^3$ and a 
$2$-dimensional direction $L_2\subseteq\R^3$ so that $L_1\cap L_2=\{\vec 0\}$.
Let $P_1,P_2:\R^3\to\R^3$ be linear maps so that $P_1(\R_3)=L_1^\perp$ and 
 $P_2(\R_3)=L_2^\perp$.
Let $S_1$ be a weakly mixing $\R^2$-action that is weakly mixing in every direction, and $S_2$ be a weakly mixing 
$\R$-action. Define $\R^3$-actions $T_1^{\vec t}=S_1^{ P_1\vec t}$ and $T_2^{\vec t}=S_2^{ P_2\vec t}$ 
so that $T_i^\vt x=S_1^{P_1\vt}x$. The actions $T_i$ are not free,
and the spectral measures $\sigma_0^{T_i}$  are wall measures consisting of $\sigma_0^{S_i}$  pushed forward  
by $\widehat P_i$ onto $L_i^\perp$. 
The product $T=T_1\times T_2$, which is free and weak mixing, has $L_1$ and $L_2$ as non ergodic directions 
 (like Example~\ref{ex:BWtype}). Its spectral measure consists of the wall measures 
 $\sigma_0^{T_1}$ and $\sigma_0^{T_2}$,
plus the convolution $\sigma_0^{T_1}*\sigma_0^{T_2}$, which has no walls. 
It follows that 
$\mathcal N\ce_T=\mathcal N\cw_T=\{L_1,L_2\}$.
\end{example}

\smallskip

\noindent {\em Comment $1$:} If $S_1$ and $S_2$ are chosen to have Lebesgue spectrum, then $\sigma_0^{T_1}*\sigma_0^{T_2}$
is equivalent to Lebesgue measure. For this $T$, any direction except $L_1$ or (a subspace of) $L_2$  is a strong mixing direction. But $T$ is not mixing.

\smallskip

\noindent {\em Comment $2$:} It is known that the spectral measure of any {\em free}
 ergodic $G$-action has support $\widehat G$ (see \cite{KT}).
In this example, $G=\widehat G=\R^3$. However, since $T_1$ and $T_2$ are not free, we see here that they can 
have wall measures 
as spectral measures.

\section{Some additional results}\label{six}

\subsection{Directional ergodicity implies directional weak mixing} 

 It follows from Corollary~\ref{whatdirections1} and Proposition~\ref{p:nwm} that if an ergodic $\Z^d$ or $\R^d$-action is 
 not weak mixing, then there exists $L\in\mathbb G_d$ so that $L\in\mathcal E_T$ and $L\in\mathcal N\mathcal W_T$.  
 The main result of this section is to show that this is not possible for weak mixing actions $T$.   
 
\begin{thm}\label{p:ergiswm}  Let $T$ be a weak mixing $\R^d$ or $\Z^d$-action and let $L\in\mathbb G_{e,d}$.  
Then $\ce_T=\cw_T$.  
 \end{thm}

For $\R^d$-actions $T$ this is a special case of Proposition~\ref{p:subgroup1}. 
Since it is clear that weak mixing implies ergodicity, we need only prove that $\ce_T\supseteq\cw_T$.
The next lemma is a key ingredient in the proof. The proposition following the lemma establishes 
the $\R^d$ case, and a little more.  
Similarly the $\mathbb Z^d$ case where $L$ is a 
completely rational direction follows from Proposition~\ref{p:subgroup1}. 

If $d=2$ and $L$ is not completely rational, then $L$ is irrational, as is $L^\perp$ (see Definition~\ref{d:rational}), and we 
can provide a quick proof of Theorem~\ref{p:ergiswm} in this case as well.   If $T$ is a weak mixing $\Z^2$-action and $T$ 
is  ergodic but not weak mixing in the direction $L$, then
by Definition~\ref{Zddef1}  $L^\perp$ is not a wall for $\varsigma_{\mathcal O}^{\widetilde T}$.  
Since $L^\perp$ is irrational, it misses all the atoms of $\varsigma_0^{\widetilde T}$ so
it is also not a wall for $\varsigma_{0}^{\widetilde T}$.  This implies that  
$\widetilde T^{|L}$ is ergodic.  By Proposition~\ref{p:subgroup1}
we have that $\widetilde T^{|L}$ is weak mixing. We conclude by Definition~\ref{Zddef1} that $T$ is weak mixing in the direction $L$.

For $d>2$, however, it is possible for $L^\perp\cap\Z^d\neq\emptyset$ even if $L$ is not rational. This means that 
$\widetilde T^{|L}$ can fail to be ergodic, but with the non-ergodicity entirely due to 
the rotation factor.

\begin{example}\label{baddir}
Let $L=\{t(0,1,\sqrt 2):t\in\R\}\subseteq\R^3$, which is not completely rational 
since $L\cap \Z^3=\{\vec 0\}$. However,  $L^\perp\cap\Z^3\sim\Z\ne\Z^{{\rm dim}(L^\perp)}=\Z^2$, so $L$ is also not irrational. 
Write $\mathbb T^3=[0,1)^3$ and note that the sets $X\times[0,\epsilon)\times[0,1)^2$ are invariant under $\widetilde T^{|L}$.  Thus we can not assume that eigenvalues of $\widetilde T^{|L}$ are simple. 
 \end{example}

 Because of Example~\ref{baddir}, the proof of Proposition~\ref{p:subgroup1} cannot be used to prove Theorem~\ref{p:ergiswm} for $\Z^d$-actions in general.  Here we provide a different proof.

\begin{proof}[Proof of Theorem~{\rm \ref{p:ergiswm}} for $\Z^d$-actions.]
 
Let $T$ be a weak mixing $\Z^d$-action, and suppose for some $0<e<d$ that $T$ is ergodic but not 
weak mixing in the direction $L\in\mathbb G_{e,d}$.   
By Definition~\ref{Zddef1}  there is  a nonzero eigenfunction  
$F\in L^2_\co(\widetilde X,\widetilde\mu)$ for $\widetilde T^{|L}$
with nonzero eigenvalue $\vec\ell\in L$. 
By an argument from the proof of Proposition~\ref{p:subgroup1}, the eigenspace
$
{\mathcal F}_{\vec \ell}:=\{F:F(\widetilde T^{|L})^{\vec t } (x)=e^{2\pi i {\vec \ell\,}\cdot \vec t}F(x)\}
$
is $U_{\widetilde T}$-invariant. Thus for any $\vec v\in\R^d$, we have that 
 $G_{\vec v}:=F\circ \widetilde T^{\vec v}  \in\mathcal F_{\vec\ell}$, and 
 therefore,  $\vert G_{\vec v}\vert$
 is $\widetilde T^{|L}$-invariant. 
 Since $T$ is direction $L$ ergodic,   
Definition~\ref{Zddef1} implies
$\vert G_{\vec v}\vert\in \mathcal O$, 
and thus for each $\vec v\in\R^d$
\eqref{e:defineo1} implies  
$\vert G_{\vec v}(x,\vec s\,)\vert = \omega_{\vec v}(\vec s\,)\in L^2([0,1)^d,\lambda)$.
If $Z=\{\vec s\in[0,1)^d: \omega_{\vec 0}(\vec s\,)=0\}$, then $\lambda(Z^c)>0$
since $F=G_{\vec 0}$ is non-zero.  
Furthermore if $\vec n\in\Z^d$, then  $\widetilde T^{\vec n}(x,\vec s)=(T^{\vec n}x,\vec s\,)$, so
$G_{\vec n}(x,\vec s\,)=F(\widetilde T^{\vec n}(x,\vec s\,))=F(T^{\vec n}x,\vec s\,)$.   
Thus for all $\vec n\in\Z^d$ the function $\omega_{\vec n}(\vec s\,)=0$ if and only if $\vec s\in Z$.

We define $H_{\vec n}(x,\vec s\,)={F(x,\vec s\,)}/{G_{\vec n}(x,\vec s\,)}$ if $\vec s\notin Z$ and 
$H_{\vec n}(x,\vec s\,)=0$ otherwise.
Since $F,G_{\vec n}\in\cf_{\vec \ell}$, we have that $H_{\vec n}$ is  $\widetilde T^{|L}$-invariant, so 
$H_{\vec n}\in \mathcal O$. Thus 
$H_{\vec n}(x,\vec s\,)=c_{\vec n}(\vec s\,)$ for some function $c_{\vec n}\in L^2([0,1)^d,\lambda)$, and
$G_{\vec n}(x,\vec s\,)=c_{\vec n}(\vec s\,)F(x,\vec s\,)$.
We claim that there is a $Z'\subseteq Z^c$, $\lambda(Z')>0$,  
so that for any $\vec s\in Z'$, 
$f_{\vec s\,}(x):=F(x,\vec s\,)$ is not constant.  
Indeed, suppose 
$f_{\vec s\,}(x)=a_{\vec s}$ for a.e $\vec s\in Z^c$.  
Then  
$F(x,\vec s\,)= 0$ for $\vec s\in Z$  and $F(x,\vec s\,)=a_{\vec s}$ for $\vec s\notin Z$.
Thus $F\in \co$, 
contradicting 
$F\in L^2_{\mathcal O}(\widetilde X,\mu)=\co^\perp$.  

Fix $\vec s\in Z'$ and define $f(x):=F(x,\vec s\,)$. 
Then 
$
f(T^{\vec n}x)=F(T^{\vec n},\vec s\,)=G_{\vec n}(x,\vec s\,)=c_{\vec n}(\vec s\,)F(x,\vec s\,)=c_{\vec n}(\vec s\,)f(x),
$
and since $T$ is measure preserving, we have $\vert c_{\vec n}(\vec s\,)\vert=1$, or in other words 
$c_{\vec n}(\vec s\,)=e^{2\pi i\alpha_{\vec n}}$ for some $\alpha_{\vec n}\in\R$.   
Arguing as in the proof of Proposition~\ref{p:subgroup1}, we have 
$e^{2\pi i\alpha_{\vec n+\vec m}}=e^{2\pi i\alpha_{\vec n}}e^{2\pi i\alpha_{\vec m}}$,
so there exists $\vec\alpha\in[0,1)^d$ so that $\alpha_{\vec n}=\vec \alpha\cdot\vec n$.
Thus $f$ is a non-zero eigenfunction for the $\Z^d$-action $T$ with eigenvalue $\vec\alpha$. This is a contradiction
since $T$ is weak mixing.
\end{proof} 

\subsection{Strong mixing}

A $G=\R^d$ or $\Z^d$-action $T$ (or a subgroup of one of these) is 
{\em strong mixing} if for each $f\in L^2_0(X,\mu)$ its correlation function 
satisfies $\widehat\sigma_f(\vec g\,)=(U_T^{\vec g}f,f)\to 0$ as 
$\vec g\to\infty$. As above, the correlation function
$\widehat\sigma_f(\vec g\,)$ is the Fourier transform of the 
spectral measure $\widehat\sigma_f$ on $\widehat G$.
A measure like $\widehat\sigma_f$, which has a Fourier transform that vanishes 
at infinity is called a {\em 
Rajchman measure}. In particular, a $G$-action $T$ is strong mixing if and only if its reduced spectral measure 
$\sigma^T_0$ is Rajchman (see \cite{KT}).  
No measure with an atom can be a Rajchman measure, so strong mixing implies weak mixing, which implies ergodicity.
By the Riemann-Lebesgue lemma, if $T$ has Lebesgue spectrum,
which means $\sigma^T_0$ is absolutely continuous, then $\sigma^T_0$ is a  
Rajchman measure,
so $T$ is strong mixing. In particular, Bernoulli actions are strong mixing, as are certain 
entropy zero actions.

It is easy to see that if $T$ is a strong mixing $\R^d$-action, then the restriction to $T^{|L}$ to any 
direction $L\in \G_d$ is also strong mixing. Indeed, strong 
mixing passes to subgroups, unlike ergodicity and weak mixing and ergodicity, which pass to supergroups. Thus,
strong mixing $\R^d$
actions are ergodic and weak mixing 
in every direction. The same argument holds for completely 
rational directions $L$ if $T$ is a $\Z^d$-action. Our goal now is to prove a the same result for $\Z^d$-
actions in arbitrary directions. 

We will need the following definition and lemma.

\begin{dfn}\label{flat}
Call a subset $I\subseteq \widehat G$, for $\widehat G=\R^d$ or $\T^d$, 
a {\em flat disc} if it is either a single point, or it is the nonempty 
intersection of an $e$-dimensional hyperplane in $G$, $0<e<d$, with an open disc 
$D\subseteq G$.
\end{dfn}

\begin{lem}\label{norajchman}
A wall measure $\sigma$ on $\widehat G=\R^d$ or $\T^d$ is not a Rajchman measure. 
\end{lem}

\begin{proof}
 If $\sigma$ is a wall measure, then $\sigma(I)>0$ for some flat disc $I$,
and it will suffice to replace $\sigma$ with the wall measure $\sigma'(E)=\sigma(E\cap I)$.
Clearly, if $\sigma'$ is atomic, then it is not Rajchman, so we assume $\sigma'$ is non-atomic. 

Without loss of generality,
we can translate $I$ so it goes through $\vec 0$. 
In the case $\widehat G=\T^d$ we can identify $\T^d=[-1/2,1/2)^d\subseteq\R^d$ and assume 
$I\subseteq [-1/2,1/2)^d$. Thus in either case, we view $\sigma'$ as a measure on $\R^d$. 
We can write any $\vec t\in \R^d$ uniquely as 
$\vec t={\vec t}_L+{\vec t}_{L^\perp}$ where  ${\vec t}_L\in L=\text{span}(I)$ and ${\vec t}_{L^\perp}\in L^\perp$.
Then since $\sigma$ is supported in $I$, 
\begin{equation*}
\widehat\sigma(\vec t\,)=\int_{\widehat G} e^{-2\pi i (\vec a\cdot \vec t\,)}d\sigma(\vec a)
=\int_I e^{-2\pi i (\vec a\cdot({\vec t}_L+{\vec t}_{L^\perp}))}d\sigma(\vec a)
=\int_I e^{-2\pi i (\vec a\cdot {\vec t}_L)}d\sigma(\vec a)=\widehat{\sigma|_L}({\vec t}_L).
\end{equation*}

Here $\sigma|_L$ is the restriction of $\sigma$ to $L$, and 
its Fourier transform 
$\widehat{\sigma|_L}({\vec t}_L)$
is a positive definite function on $L$.
Thus $\widehat\sigma(\vec t\,)$ is constant on each coset $L^\perp+{\vec t}_L$
and in particular, is not 0 at infinity.  
\end{proof}

\begin{cor}\label{sw}
If $T$ is a strong mixing $G=\R^d$ or $\Z^d$-action then $T$ is weak mixing  in every direction, and thus ergodic in every direction. 
\end{cor}

\begin{proof}
Having already discussed $G=\R^d$,
we just need to prove the case $G=\Z^d$. 
Since $T$ is strong mixing, $\sigma_0^T$ is a Rajchman measure 
on $\widehat G=\T^d$. If $T$ is not weak mixing in the direction $L$, then
$L$ is a wall and the corresponding wall measure $\sigma'$ satisfies 
$\sigma'\ll  \sigma_0^T$. But this is impossible since the Rajchman property 
is preserved by absolute continuity (Proposition 2.5 in \cite{KT}). 
\end{proof}

For an $\R^d$-action $T$ and a direction $L\in \G_{e,d}$, we say $T$ is 
{\em strong mixing in the direction $L$} if $T^{|L}$ is mixing. 
Following our 
usual practice, we say a $\Z^d$-action $T$ is
{\em strong mixing in a direction 
$L\in \G_{e,d}$}
for any 
$f\in L^2_{\mathcal O}( \widetilde X,\widetilde \mu)$, the unit suspension 
correlation function satisfies
${\widehat \varsigma\,}_f^{\widetilde T}(\vec t\,):=(U_{\widetilde T}^{\vec t}f,f)\to 0$ as $||\vec t\,||\to\infty$ in $L$.

\begin{prop}
If $T$ is a strong mixing $G=\R^d$ or $\Z^d$-action, then $T$ is strong mixing in every direction.
\end{prop}
\begin{proof}
We noted above that if $T$ is strong mixing $\R^d$-action then each $T^{|L}$ is strong mixing,
thus  
$T$ is strong mixing in every direction. 
Now assume 
$T$ is a strong mixing $\Z^d$-action. Then 
$\sigma_0^T$ is a 
Rajchman measure on $\widehat G=\T^d$.
To show that $T$ is mixing in any direction $L\in\mathbb G_{e,d}$ we show the stronger result that 
for any $f\in L^2_{\mathcal O}(X,\mu)$,
$\varsigma_f^{\widetilde T}$ is a Rajchman measure,
since then 
${\widehat \varsigma\,}_f^{\widetilde T}(\vec t\,)\to 0$ for any $||\vec t||\to\infty$. 
Proposition 2.5 in \cite{KT} shows that the Rajchman property is preserved by 
absolute continuity.  
Since $\varsigma_f^{\widetilde T}\ll\varsigma_\co^{\widetilde T}$ and
by Theorem~\ref{unit1}, $\varsigma_\co^{\widetilde T}\sim\widetilde{\sigma_0^T}$ it suffices 
to show $\widetilde{\sigma_0^T}$ 
is a Rajchman measure on $\R^d$.
Taking the Fourier transform of (\ref{average1})  over $\R^d$, we have
\begin{equation*}
(\widetilde{\sigma_0^T})\widehat{\phantom{l}}(\vec t\,)
=\sum_{\vec n\in \Z^d} a_{\vec n}(\sigma_0^T\circ\tau^{-\vn})\widehat{\phantom{o}}(\vt\,)
=\sum_{\vec n\in \Z^d} a_{\vec n}\left(e^{2\pi i (\vec n\cdot\vec t\,)}\,\widehat{\sigma_0^T}(\vec t\,)\right)
=\left(\sum_{\vec n\in \Z^d} a_{\vec n}\,e^{2\pi i (\vn\cdot\vt\,)}\right)
\widehat{\sigma_0^T}(\vec t\,),
\end{equation*}
where 
\begin{equation*}
\widehat{\sigma_0^T}(\vec t\,)=\int_{\R^d}e^{-2\pi i({\vec \ell}\cdot\vt)}d\sigma_0^T({\vec \ell}\,)
=\int_{[0,1)^d}e^{-2\pi i({\vec \ell}\cdot\vt)}d\sigma_0^T({\vec \ell}\,).
\end{equation*}
is the Fourier transform of $\sigma_0^T$ over $\R^d$, so we 
need to show that $\sigma_0^T$ is a Rajchman measure on $\R^d$.    
By the mixing assumption we know that 
$\sigma_0^T$ is a Rajchman measure on $\T^d=[0,1)^d$,
which means 
\begin{equation*}
\widehat{\sigma_0^T}(\vec n\,)=\int_{[0,1)^d}e^{-2\pi i({\vec \ell}\cdot\vn)}d\sigma_0^T({\vec \ell}\,)\to 
0\text{ as }||\vn||\to \infty\text{ in }\Z^d.
\end{equation*}
We need to show the same happens for $||\vec t||\to\infty\text{ in }\R^d$.  
 
Take a sequence $\vt\in\R^d$ with  $||\vt_j||\to \infty$ and suppose $|\widehat{\sigma_0^T}(\vec t_j)|\ge K>0$.  Write $\vt_j=\vn_j+\vw_j$ where $\vn_j\in \Z^d$ and $\vw_j\in [0,1)^d$.
By passing to a subsequence we can assume $\vw_j\to\vw\in[0,1]^d$.  Define a measure $\sigma'$ by setting 
$d\sigma'({\vec \ell}\,)=e^{-2\pi i ({\vec \ell}\cdot\vw)}d\sigma_0^T(\vec \ell\,)$.  Note that $\sigma'\ll\sigma_0^T$ and
 \begin{align*}
\widehat{\sigma'}(\vn_j)=\int_{[0,1)^d}e^{-2\pi i ({\vec \ell}\cdot\vn_j)}d\sigma'(\vec \ell\,)
&=\int_{[0,1)^d}e^{-2\pi i ({\vec \ell}\cdot(\vn_j+\vw))}d\sigma_0^T(\vec \ell\,).
\end{align*}
Let  $\epsilon<\frac{K}2$, then 
\begin{align*}
|\widehat{\sigma'}(\vn_j)-\widehat{\sigma_0^T}(\vec t_j)|&=\left|\int_{[0,1)^d}e^{-2\pi i ({\vec \ell}\cdot(\vn_j+\vw))}-e^{-2\pi i(\vec\ell\cdot\vec t_j)}d\sigma_0^T(\vec \ell\,)\right|\\
&=\left|\int_{[0,1)^d}e^{-2\pi i (\vec\ell\cdot\vec t_j)}\left(e^{-2\pi i(\vec\ell\cdot(\vec n_j+\vec w-\vec t_j))}-1\right)d\sigma_0^T(\vec \ell\,)\right|\\
&=\left|\int_{[0,1)^d}e^{-2\pi i (\vec\ell\cdot\vec t_j)}\left(e^{-2\pi i(\vec\ell\cdot(\vec w-\vec w_j))}-1\right)d\sigma_0^T(\vec \ell\,)\right|\\
&\le \max_{\vec\ell\in[0,1)^d}\left|e^{-2\pi i(\ell\cdot(\vec w-\vec w_j))}-1\right|\sigma_0^T([0,1)^d)<\epsilon
\end{align*}
for $j$ sufficiently large.  It follows that $\vert\widehat\sigma'(n_j)\vert\ge \frac{K}2$ for $j$ large enough so it is not Rajchman on $\mathbb Z^d$, but yet $\sigma'\ll\sigma_0^T$, a contradiction.  
\end{proof}

A  similar argument shows the following. 

\begin{prop}
If $T$ is 
strong mixing in a direction $L$ then it is weak mixing and ergodic in the direction $L$, and in particular, $T$ is ergodic.
\end{prop}

\subsection{Rigidity and Genericity}\label{s:rigidity}
In this section we investigate the prevalence of directional weak mixing and its relationship to rigidity.

\begin{dfn}\label{d:rigidl}
A $\Z^d$-action $T$ is {\it rigid}
if there exists a sequence $\vec n_k\in\Z^d$, called a {\em rigidity sequence} for $T$, so that $T^{\vec n_k}\to \text{Id}$ in the weak topology, i.e., $||T^{\vec n_k}f-f||_2\to 0$ for all $f\in L^2(X,\mu)$.  
For completely rational directions $L$ (so $\Lambda:=L\cap \Z^d\sim \Z^e$ where $e=\dim(L)$), we say $T$ is {\em rigid in the direction $L$} if we can choose all the $\vec n_k\in\Lambda$.
\end{dfn}

Note that rigidity in any direction $L$ implies $T$ is rigid.
We begin by showing that rigidity in a completely rational direction is compatible with weak mixing in that direction.    

\begin{prop}\label{WMDReg}
There is a rigid
weak mixing $\Z^d$-action $T$ that is rigid in the vertical, horizontal and diagonal directions, and is weak 
mixing in all directions.
\end{prop}

For simplicity we give the proof in the case $d=2$. 
The proof relies on the following well known characterization of rigidity (see for example Proposition 2.5, (5) and Corollary 2.6 in \cite{BdJLR}), and some classical results about Kronecker sets.

\begin{prop}\label{p:maxrigid}
A sequence $\{\vec n_k\}\subseteq\Z^2$ is a rigidity sequence for $T$ if and only if $f_k(\vec \ell):=e^{2\pi i(\vec n_k\cdot \vec \ell\,)}\rightarrow 1$ in $L^2(\mathbb T^2,\sigma_0^T)$.
\end{prop}

\begin{dfn}\label{Kroneckerset}
A compact subset $K\subseteq\mathbb R$ is a {\it Kronecker set} if every continuous function 
$\varphi:K\rightarrow\mathbb C$ with $\vert\varphi(a)\vert=1$ can be uniformly approximated on $K$ by characters.   
\end{dfn}

We state, without proof,  the following result due to Wik \cite{Wik}.
\begin{prop}\label{p:Wik}
There is a Kronecker set $K \subseteq [0,1]$ 
with full Hausdorff dimension.  
\end{prop}
For more details on Kronecker sets see \cite{Rudin} and \cite{Kaufman} for a different proof of Proposition~\ref{p:Wik}.

\begin{proof}[Proof of Proposition~{\rm \ref{WMDReg}}]
We will use the Gaussian Measure Construction (GMC),  as in the proof of Theorem~\ref{t:cecomplete}, to obtain the required $\mathbb Z^2$-action $T$.  
We begin by describing the measure $\sigma$ that we use in the construction.

Let $K\subseteq[0,1]$ be a Kronecker set of full Hausdorff dimension and consider the product $K\times K$.   The Hausdorff dimension of $K\times K$ is 2 (see for example \cite{Hatano}).   By Frostman's Lemma \cite{Mattila} this means that for any $\delta > 0$  there is a  finite Borel measure $\sigma$ on $K\times K$ such that for all $\vec a\in \R^2$ and $r > 0$,
$\sigma(\{\vec b \in \R^2: \|\vec b-\vec a\|_2 \le r\})\le r^{2-\delta}$.
Letting 
$\delta=\frac12$ it follows that $\sigma$ vanishes
on all line segments.  But then, the same holds for any convolution power $\sigma^{(k)}=\sigma*\sigma*\dots*\sigma$, $k$ times, with $k \ge 2$.

Now 
just as in Theorem~\ref{t:cecomplete}, 
the GMC with measure $\sigma$ yields an ergodic  
$\Z^2$-action $T$  on a Lebesgue probability space
$(X,\mu)$ such that $\sigma^T_{0}= \sum_{n=1}^\infty \sigma^{(n)}/n!$ is the restricted spectral measure for $T$.  Since 
$\sigma^T_{0}$ vanishes on all line segments, by Theorem~\ref{XXZ2}
$T$ is weak mixing in all directions.  

Fix $\epsilon>0$. By Definition~\ref{Kroneckerset}, there exist $m\in\mathbb Z$ so that  for all $a\in K$,
$|e^{ima}-1| \le \epsilon/2$.   Hence, we can choose vectors $\vec n_k\in\mathbb Z^2$ of the form  $(0,\pm n_k)$, $(\pm n_k,0)$, or $(\pm n_k,\pm n_k)$ such that 
$e^{2\pi i({\vec n}_k\cdot\vec a\,)} \to 1$ 
as  $k \to \infty$  for all $\vec a=(a_1,a_2)\in K\times K$.
It follows that $e^{2\pi i({\vec n}_k\cdot\vec a\,)}\to 1$ in $L^2(\mathbb T,\sigma)$ and thus in $L^2(\mathbb T,\sigma_0^T)$.Therefore by Proposition~\ref{p:maxrigid}, 
$T$ has rigidity sequences $\{\vec n_k\}$ of the form $\{(0,\pm n_k)\}$, 
$\{(\pm n_k,0)\}$, and $\{(\pm n_k,\pm n_k)\}$.  
\end{proof}
 
We note that because we use the GMC, the $\Z^2$-action constructed in the proof of 
Proposition~\ref{WMDReg} embeds in a rigid, weak mixing $\R^2$-action $\overline T$ which is, 
itself, weak mixing in all directions. We note that it is not known if there is a rank one action that 
satisfies Propostion~\ref{WMDReg}. 

\bigskip

Let $\mathcal A_d$ denote the set of all measure preserving $\Z^d$-actions on a Lebesgue probability space $(X,\mu)$. 
Next, we are going to consider directional weak mixing and rigidity for $\Z^d$-actions in the Baire category sense.
To do this, we consider two topologies, the {\em weak topology} and the {\em uniform topology} on the set
$\mathcal A_d$.
These  ideas, are well known for $\Z$-actions (see \cite{Halmos});
for $\Z^d$-actions, a general reference is 
\cite{kechris}.  The definitions and some basic results are mentioned here 
for convenience.

\begin{dfn}
The {\em weak topology} on the set $\mathcal A_d$  is defined by
$T_n\to T$ if $U_{T_n}^{\vec n}f\to U_{T}^{\vec n}f$ for 
all $\vec n\in \Z^d$ and $f\in L^2(X,\mu)$, or equivalently, 
$\mu(T_n^{\vec n}E\triangle T^{\vec n}E)\to 0$ for all measurable $E\subseteq X$.
The {\em uniform topology} on the set $\mathcal A_d$  is defined by
$T_n\to T$ if $\lim_n\sup_{E\subseteq X}\mu(T_n^{\vec n}E\triangle T^{\vec n}E)=0$
for all $\vec n\in \Z^d$. 
\end{dfn}

Clearly, uniform convergence implies weak convergence. 
The weak and uniform topologies on $\mathcal A_d$ are both completely metrizable
(see \cite{kechris}), but the weak topology, $\mathcal A_d$ is also separable 
(see \cite{kechris}), so Polish. The  
metrics that give the uniform and weak topologies on $\mathcal A_d$
depend on the corresponding metrics for the 
the group $\mathcal G:=\mathcal A_1$ of all
invertible measure preserving transformations $R$ on $(X,\mu)$.
The uniform topology is given by the
metric 
\begin{equation}\label{d1}
\delta_1(T,S)=\mu\{x\in X:Tx\ne Sx\},
\end{equation}
 (see \cite{Halmos}). For the weak topology, 
let $\{E_k:k\in\N\}$ be collection measurable sets in $X$, dense in the  
metric $\mu(E\triangle E')$. 
Then the weak topology is given by the metric
$\rho_1(T,S)=\sum_k (\mu(TE_k\triangle SE_k)+\mu(T^{-1}E_k\triangle S^{-1}E_k))$
for $T,S\in\mathcal G$.
For $\vn=(n_1,n_2,\dots,n_d)\in\Z^d$ define $||\vn||_\infty=\max_{i=1}^m |n_i|$. 
Then for $T,S\in\mathcal A_d$, the weak topology is given by the 
complete metric $\rho_d(T,S)=\sum_{\vn\in\Z^d}2^{-||\vn||_\infty}\rho_1(T^{\vn},S^{\vn})$
and the uniform topology is given by metric
\begin{equation}\label{ut}
\delta_d(T,S)=\sum_{\vn\in\Z^d}2^{-||\vn||_\infty}\delta_1(T^{\vn},S^{\vn}),
\end{equation}
(see \cite{kechris}).
For the case of $\R^d$-actions, the technicalities 
are more complicated, but we hope to address them 
in a later paper.

We call a set $\mathcal B\subseteq\mathcal A_d$ a {\em property}, (i.e., a  $\Z^d$-action $T$
has the property if $T\in \mathcal B$). For $T\in\mathcal A_d$ and $R\in\mathcal A$, 
define the $\Z^d$-action $S=RTR^{-1}=\{RT^{\vec n}R^{-1}\}_{\vec n\in\Z^d}\in\mathcal A_d$ 
isomorphic to $T$. We are usually most interested in {\em isomorphism invariant} properties: 
$RTR^{-1}\in\mathcal B$ whenever $T\in\mathcal B$.  Properties like ergodicity, weak mixing, 
and directional ergodicity or weak mixing in a particular set of directions are isomorphism invariant properties. 
Another example is 
the {\em isomorphism class} $\mathcal C_T:=\{RTR^{-1}:R\in\mathcal A\}$ 
of any $T\in\mathcal A_d$. 
We say a property $\mathcal B$ is {\em generic} (or that {\em a generic $\Z^d$-action $T$
has the property}) 
if $\mathcal B\subseteq \mathcal A_d$ is  residual (usually, we show it has a  
dense $G_\delta$ subset).

Ryzhikov \cite{Ryz} showed that a generic $\Z^d$-action $T$ is weak mixing, and also 
that  a generic $\Z^d$-action $T$
embeds in an $\R^d$-action $\overline T$
such that   $\overline T^{|L}$ is weak mixing for every $L\in\mathbb G_d$.  Our Corollary~\ref{t:wmemb}, 
says that whenever a weak mixing $\Z^d$-action $T$ embeds 
in an $\R^d$-action $\overline T$, the $\Z^d$-action  $T$ and the $\R^d$-action 
$\overline T$ have the same weak mixing directions.
Combining these two, we see  
that weak mixing in all directions is generic for $\Z^d$-actions $T$.
Our goal now is to give a direct proof of this.

\begin{thm} \label{generic}
Let $\mathcal W\mathcal A_d$ be the set of all weakly mixing $T\in\mathcal A_d$ that are weakly mixing in all directions
(an isomorphism invariant property).  
Then  
$\mathcal W\mathcal A_d$ is a dense $G_\delta$ subset of $\mathcal A_d$
in the weak-topology.
\end{thm}

The remainder the paper consists of the proof of Theorem~\ref{generic}.
It should be noted that 
Glasner and King \cite{GlasnerKing} show that if $\mathcal B$ is an isomorphism invariant property 
 with the additional assumption that $\mathcal B$ is a Baire set, 
then either $\mathcal B$ or its complement is generic. If we knew
$\mathcal W\mathcal A_d$ was a Baire set, this would simplify the proof below. 
However, we do 
not know whether or not this is the case.

The proof of Theorem~\ref{generic} relies on the following proposition, 
which is well know for $\Z$-actions (see e.g., \cite{Halmos} or \cite{KT}). 
We include a $\Z^d$ proof here for completeness.
This is essentially a corollary of the Rohlin lemma for $\Z^d$ (see \cite{KatzWeiss} or \cite{Conze}),
which says if $T$ is free, then for any $\epsilon>0$ and $m\ge 0$, with 
$Q_m:=\{\vn\in\Z^d:||\vn||_\infty\le m\}$,  there exists a 
set $B\subseteq X$,
called a shape-$Q_m$ {\em Rokhlin tower base}, 
such that 
$T^\vn B$ for  $\vn\in Q_m$ are pairwise disjoint, and $\mu(T^{Q_m}B)=1-\epsilon$.

\begin{prop}\label{densconj}
Let $\mathcal  A_d$ be the set of all measure preserving $\Z^d$-actions $T$  on $(X,\mu)$.
If $T\in\mathcal A'_d$ (i.e., $T$ is free) then the set 
$\mathcal C_T$  is dense in $\mathcal A_d$ in the weak topology.
\end{prop}

\begin{proof}
Let $S\in \mathcal A'_d$ and fix $\delta>0$. 
We claim that there is $R\in\mathcal G$ 
so that  $\delta_d(S,RTR^{-1})<\delta$.
This shows that $\mathcal C_T$ is uniformly dense
in $\mathcal A'_d$.
But since uniform convergence implies weak convergence,
it follows that $\mathcal C_T$ is weakly dense
in $\mathcal A'_d$. 
Finally, Glasner and King \cite{GlasnerKing} show
that $\mathcal A'_d$ is dense $G_\delta$ in $\mathcal A_d$,
and the lemma follows.

To prove the claim, choose
$K$ so large that 
$\sum_{||\vn||_\infty>K}2^{-||\vn||_\infty}<\delta/2$.
Let $\gamma<(\delta/2)/(2K+1)^d$, and let
$m$ be large enough that $\rho:=(2m-d K)/(2m+1)>1-\gamma/2$,
and let $\epsilon<\gamma/2$. 

Since $T,S\in \mathcal A'_d$, we can apply the Rohlin lemma for $\Z^d$ to 
find shape-$Q_m$
Rohlin tower bases for $T$ and $S$ with
$\mu(T^{Q_m}A)=1-\epsilon$, and $\mu(T^{Q_m}B)=1-\epsilon$. 
Since 
$\mu(A)=\mu(B)=(1-\epsilon)/(2m+1)^d$, we can  define a map 
$R(B)=A$, bijective and measure preserving,  
but otherwise arbitrary. Then
we extend 
$R$ to be a measure preserving map $R:S^{Q_m}B\to T^{Q_m}A$ by 
$R(S^{\vec n} x)=S^{\vec n}(Rx)$ for $x\in B$ and $\vec n\in Q_m$.
Finally, we define $R:(S^{Q_m}B)^c\to (S^{Q_m}A)^c$, again
bijective and measure preserving, but otherwise arbitrary. 
Clearly $R\in\mathcal G$. Let $T_1=RTR^{-1}$. By the definition of $R$, 
$B$ is a shape-$Q_m$ Rokhlin tower base for both $S$ and $T_1$.
 
Now suppose $||\vn||_\infty\le K$. Then $(Q_m\cap(Q_m-\vn))+\vn\subseteq Q_m$.
Thus if $x\in S^{Q_m\cap(Q_m-\vn)}B$, then $S^\vn x=T_1^\vn x$
since $B$ is a shape $Q_m$ Rokhlin tower base for both $S$ and $T_1$.
By (\ref{d1}) it follows 
that
\begin{equation}\label{size}
\delta_1(S^\vn,T_1^\vn)\le 1-\mu(S^{Q_m\cap(Q_m-\vn)}B).
\end{equation}
Now $\#(Q_m)=(2m+1)^d$, and by the binomial theorem
\begin{align*}
\#\left(Q_m\cap(Q_m-\vn)\right)&\ge (2m+1)^{d-1}(2m-d||\vn||_\infty)\\
&\ge (2m+1)^{d-1}(2m-d K).
\end{align*}
Thus 
\begin{align*}
\mu(S^{Q_m\cap(Q_m-\vn)}B)&\ge\mu(B)(2m+1)^{d-1}(2m-d K)\\
&\ge (1-\epsilon)\rho>(1-\gamma/2)^2>1-\gamma.
\end{align*}
 
By (\ref{size}),
$\delta_1(S^\vn,T_1^\vn)<\gamma$ 
for all
$||\vn||_\infty\le K$.
So
(\ref{ut}) implies
\begin{align*}
\delta_d(S,RTR^{-1})=\delta_d(S,T_1)&=\sum_{\vn\in\Z^d}2^{-||\vn||_\infty}\delta_1(S^{\vn},T_1^{\vn})\\
&\le\delta/2+\sum_{||\vn||_\infty\le K}\delta_1(T^{\vn},T_1^{\vn})\\
&\le\delta/2+ (2K+1)^d \gamma<\delta.
\end{align*}
This proves the claim.
\end{proof}

In what 
follows, for $T\in\mathcal A_d$ and $f\in L^2(X,\mu)$, we write  $\sigma_f^T$ 
for the spectral measure 
of $f$ with respect to  the action $T$, a measure on $\T^d=\widehat{\Z^d}$.
Also, let $\mathcal I$ denote the set of all the set of all 
flat discs in $\T^d$ (see Definition~\ref{flat}).

\begin{lem}\label{l:closednoint}  Let $f\in L_0^2(X,\mu)$ and $k \ge 1$. Define 
\begin{equation*}
\mathcal F_k(f):=\{T\in\mathcal A_d:\sigma^T_f(I)\ge 1/k\text{ for some } I\in\mathcal I\}.
\end{equation*}
Then $\mathcal F_k(f)$ is a closed subset of $\mathcal A_d$ with no interior.
\end{lem}

\begin{proof} 
Let $T_n\in \mathcal F_k(f)$ so that $T_n\to T$ in $\mathcal A_d$ in the weak topology. 
Equivalently, $U_{T_n}^{\vec n} f\to U_T^{\vec n} f$  for all 
$f\in L^2_0(X,\mu)$.   
Since $\widehat{\sigma_f^T}(\vec n\,)=(U_T^{\vec n} f,f)$ 
it follows that 
\[\int_{\T^d}e^{-2\pi i (\vn\cdot\vt\,)}d\sigma_f^{T_n}\to \int_{\T^d}e^{-2\pi i (\vn\cdot\vt\,)}d\sigma_f^{T},\]
which, by the Stone-Weierstrass theorem implies $\sigma_f^{T_n}\to\sigma_f^T$ in the weak* topology.   

Since $T_n\in \mathcal F_k(f)$, we have that
$\sigma_f^{T_n}(I_n) \ge1/k$ for some $I_n \in \mathcal I$.  By passing to a subsequence, we can 
assume that there is a flat disc $I$ with the property that for all 
$U\supseteq I$ open,
there is an open set $U'$ with $I\subseteq U'\subseteq U$ so that $I_n\subseteq U'$ for all sufficiently 
large $n$.  Thus for such $n$ we have $\sigma^{T_n}_f(U')\ge1/k$.  Let $h\in C(\T^2)$ satisfy 
$h=1$ on $U'$ and $h=0$ on $U^c$. Then for all $\epsilon>0$, and sufficiently large $n$, we have
\[
\frac1k\le\sigma^T_f(U')\le\int_{\T^d} h\ d\sigma_f^T\le\int_{\T^d} h\ d\sigma_f^T +\epsilon.
\]
Letting $\epsilon$ go to zero and noting that $\chi_U\ge h$ we have
\[
\sigma_f^T(U)\ge\int_{\T^d} h\ d\sigma_f^T\ge 1/k.
\]
Using the outer regularity of $\sigma$, we conclude that $\sigma_f^T(I)\ge 1/k$, so 
$T\in\mathcal F_k(f)$.

To see that $\mathcal F_k(f)$ 
has no interior, we note that the strong mixing actions $T$ are dense in 
$\mathcal A_d$. In particular, if $T$ is strong mixing then $T$ is free 
(i.e., $T\in\mathcal A'_d$), so by Proposition~\ref{densconj},  
$\mathcal C_T=\{RTR^{-1}:R\in \mathcal G\}$ is dense in $\mathcal A_d$. 
But every $T'\in  \mathcal C_T$ is isomorphic to $T$, so $T'$ is strong mixing.
Thus there is a dense set of $T\in\mathcal A_d$ so that 
$T\notin \mathcal F_k(f)$ for any $k$.
\end{proof}

\begin{proof}[Proof of Theorem~{\rm \ref{generic}}] Fix a dense sequence $\{f_n\}\subseteq L^2_0(X,\mu)$.  
Suppose $T\notin\mathcal W\mathcal A_d$, i.e. $T$ is not weakly mixing in some direction.  
Then there is $f\in L^2_0(X,\mu)$ and $k \ge 1$ such that the spectral measure $\sigma_f^T$ 
has $\sigma_f^T(I) \ge 1/k$ for some $I\in \mathcal I$.  Approximating $f$ by the dense sequence, 
we  have that there is some $f_n$ with $\sigma_{f_n}^T(I) \ge 1/(2k)$. 
In particular, 
$\mathcal W\mathcal A^c_d \subseteq \bigcup\limits_{k=1}^\infty \bigcup\limits_{n=1}^\infty \mathcal F_k(f_n)$, 
 so by Lemma~\ref{l:closednoint}, $\mathcal W\mathcal A_d$ contains dense a  $G_\delta$. 
\end{proof}

\section{Further directions}

In this paper, we have described a general theory of directional ergodic properties, as well as examples 
$T$ of both $\Z^d$-
and $\R^d$-actions that realize certain sets of $\ce_T$ and $\cw_T$. The most interesting cases are when 
$T$ is weak mixing because otherwise there can be no weak mixing directions (Proposition~\ref{p:nwm}). If $T$ is strong mixing (which implies weak mixing),  
then $T$ is strong mixing in every direction, and thus weak mixing and ergodic  in every direction
(Colollary~\ref{sw}). Mixing examples $T$ include Bernoulli actions, which are positive entropy, 
but also zero entropy examples like the  Ledrappier's $\Z^2$-action \cite{LDr} which is strong (but not multiple) mixing,
 and even has Lebesgue spectrum.

A more interesting case is when $T$ is weak mixing but not strong mixing. Theorems~\ref{whatdirections} and 
\ref{t:cecomplete} completely 
characterize what sets of weak mixing directions are possible in this situation. 
However, our Gaussin examples in Theorem~\ref{t:cecomplete} (as well as the
generalized Bergelson-Ward actions mentioned in Example~\ref{gbw}), somehow seem a little  artificial.
Probably the most interesting open question here is what sort of directional properties 
more natural types of weak but not strong mixing $\Z^d$ and $\R^d$-actions $T$ can have?
One natural class of examples of $\R^2$-actions to look at are the
weakly mixing finite local complexity (FLC) substitution tiling  systems, because 
these actions $T$ are known never to be strong mixing (\cite{Sol}).
(The Penrose tiling dynamical system, Example~\ref{pentil} is an FLC substitution, but not weak mixing, and 
the pinwheel tiling dynamical system, Example~\ref{pintil}, is not FLC). 
We conjecture that weakly mixing FLC substitution tiling $\R^2$-actions $T$ 
are (typically) weak mixing in every direction. In fact by Theorem~\ref{p:ergiswm}, it would be 
enough to know they were ergodic in every direction.  

An interesting related set of questions concerns the property of {\em minimal self joining} (MSJ) for actions $T$ that are 
weak mixing but not mixing. Examples of MSJ actions include the Chacon $\Z^2$-action $T$, discussed in \cite{PR}, and 
the 
closely related Chacon $\R^2$-action $T$ in \cite{RC} (which is actually a fusion tiling dynamical system). 
The Chacon $\Z^2$-action $T$ is 
weak mixing in every (completely) rational direction \cite{PR}, which is a consequence of its MSJ property. 
It is easy to extend this proof to the Chacon $\R^2$-action $T$
to show it is weak mixing in every direction.
A natural conjecture, which is probably not too difficult, 
is that a $\Z^d$-action with MSJ is weak mixing in every direction (rational or irrational). 
In particular, this would apply to the Chacon $\Z^2$-action.

As we have noted in the introduction, there are typically two ways to define directional properties, 
and the one we have chosen in this paper is via the unit suspension. But unit suspensions are not the only way to study 
the direction $L$ properties of a $\Z^d$-action $T$. One can also  
study transformations $\{T^{\vec n_i}\}$ for sets of vectors 
$\vec n_i\in\Z^d$ that approximate $L$. 
Recently, Liu and Xu (\cite{LX},\cite{LX2}) have defined directional properties this way.
As we have observed, the two ways to study directional properties agree in some cases
(\cite{PIsrael},  \cite{BL} and  \cite{JSRec}).  However in other  
cases they may disagree (see \cite{Danrec}). The relationship between this work and
that of Liu and Xu remains to be studied. 

Finally, there are many open questions that have to do with directional rigidity, as discussed in Section~\ref{s:rigidity}.
Directional rigidity for any $L\in\G_d$ seems to make sense for an $\R^d$ action $T$ by defining  it to mean 
$T^{\vec v_i}\to I$, in the weak topology, 
for $\vec v_i\in L$. Directional rigidity for $L$ then implies directional rigidity $L'\supseteq L$,
and in particular, rigidity. The same idea works for rational directions in $\Z^d$ as already noted in Section~\ref{s:rigidity}.
One idea for defining directional rigidity for a $\Z^d$-action $T$, in a not necessarily rational direction $L$,
which is keeping with the philosophy of this paper, is to use 
the Koopman representation $U_{\widetilde T}$ 
restricted to $L^2_\co(\widetilde X,\widetilde \mu)$. 
The definition would be $U_{{\widetilde T}^{\vec v_i}}\to I$ for $\vec v_i\in L$. 
 In both cases ($\R^d$ and $\Z^d$) directional rigidity is clearly a spectral property,
 but we don't know much else about it. 
 Another approach would be (as in the 
 the previous paragraph) to define directional rigidity to mean 
that $T^{\vec n_i} \to I$  for  vectors 
$\vec n_i\in\Z^d$ that approximate $L$ in some way (e.g., $||n_i-L||\to 0$ as $i\to\infty$).

\def\cprime{$'$}

\end{document}